\documentclass[12pt,psfig,reqno]{amsart}
\usepackage{bbm}
\usepackage{}
\usepackage{amssymb}
\usepackage{amsfonts}
\usepackage{amscd}
\usepackage{amssymb,amsfonts,latexsym}
\usepackage{graphics,verbatim}
\usepackage{graphicx}
\usepackage{subfigure}
\usepackage[usenames]{color} 
\makeatletter

\renewcommand\section{\@startsection {section}{1}{\z@}%
                                   {-3.5ex \@plus -1ex \@minus -.2ex}%
                                   {2.3ex \@plus.2ex}%
                                   {\centering\normalfont\bf}}

 \numberwithin{equation}{section}
\numberwithin{equation}{section}

\numberwithin{equation}{section}

\theoremstyle{plain}
\newtheorem{thm}{Theorem}[section]
\newtheorem{lemma}[thm]{Lemma}
\newtheorem{pro}[thm]{Proposition}

\newtheorem*{thm*}{Theorem}
 \newtheorem{Coj}[thm]{Conjecture}
\begin{document}
\title{On the connectedness of planar self-affine sets}
\author{Jing-Cheng Liu$^{1}$, Jun Jason Luo$^{2}$ and Heng-wen Xie$^{1}$}
\address{1 Key Laboratory of High Performance Computing and Stochastic Information Processing (Ministry of Education of China), College of Mathematics and
Computer Science, Hunan Normal University, Changsha, Hunan 410081, China} \email{liujingcheng11@126.com}
\address{2 Department of Mathematics, Shantou University, Shantou 515063, Guangdong, P.R. China}\email{jasonluojun@gmail.com}\email{luojun@stu.edu.cn}

\date{\today}
\keywords {Connectedness, digit set, self-affine set, self-affine tile, neighbor.}
\subjclass[2010]{Primary 28A80; Secondary 52C20, 52C45.}
\thanks{ The research is supported in part by the NNSF of China (No.11171100, No.11301175, No.11301322), the Hunan Provincial NSF (No.13JJ4042),
the Special Fund for Excellent Doctoral Dissertation of Hunan Provincial (No.YB2012B025), Specialized Research Fund for the Doctoral Program of Higher Education of China (20134402120007), Foundation for Distinguished Young Talents in Higher Education of Guangdong (2013LYM\_0028) and STU Scientific Research Foundation for Talents (NTF12016).}

\begin{abstract}  In this paper, we consider the connectedness of planar self-affine set $T(A,\mathcal{D})$ arising from an integral expanding
matrix  $A$ with characteristic polynomial $f(x)=x^2+bx+c$ and a digit set $\mathcal{D}=\{0,1,\dots, m\}v$. The necessary and sufficient conditions only depending on $b,c,m$ are  given for the $T(A,\mathcal{D})$ to be connected. Moreover, we also consider the case that ${\mathcal D}$ is non-consecutively collinear.
\end{abstract}

\maketitle

\section{\bf Introduction}

Let  $A\in M_n({\mathbb{Z}})$  be an expanding $n\times n$ integral matrix, i.e., all eigenvalues of $A$ have moduli  strictly greater than $1$. Let $\mathcal{D} = \{d_1, \dots , d_m\}$ be a finite set of $m$ distinct vectors on ${\mathbb R}^n$. We call ${\mathcal D}$  a {\it digit set}. Then the maps
\begin{equation*}
S_i(x)= A^{-1}(x + d_i),\quad  1\leq i\leq m,
\end{equation*}
are  contractive under a suitable norm in $\mathbb{R}^n$ \cite{LW},  and it is well-known that there exists a unique non-empty compact set $T:=T(A,\mathcal{D})$ satisfying the set-valued functional equation
\begin{equation}\label {(1.2)}
T=\bigcup_{i=1}^m S_i(T).
\end{equation}
Usually $T$ can also be written as
\begin{equation*}
T= A^{-1}(T+{\mathcal D}) =\left\{\sum_{i=1}^\infty A^{-1}d_{j_i}: d_{j_i}\in \mathcal{D}\right\}.
\end{equation*}
The $T$ is called the {\it self-affine set (or attractor)} of the iterated function system (IFS) $\{S_i\}_{i=1}^m$.  We call $T$ a self-affine tile if it has positive Lebesgue measure and the union in \eqref {(1.2)} is essentially disjoint, i.e., the intersection $(T + d_i)\cap(T + d_j)$ has zero Lebesgue measure for $i\neq j$. In this situation, $T^\circ\neq \emptyset$ and $c:=|\det(A)|=m$.

The study on the topological properties of self-affine sets/tiles $T(A, {\mathcal D})$ on ${\mathbb R}^n$ is an interesting topic in fractal geometry, tiling theory and even canonical number systems. In particular, connectedness, disk-likeness (i.e., homeomorphic to a closed disk in the case $n=2$) and the boundary structures have been investigated extensively.  Gr\"{o}chenig and Haas \cite{GH} as well as Hacon {\it et al.} \cite{HSV} discussed a few special connected self-affine tiles.  Lau and his coworkers (\cite{HKL}, \cite{KL}, \cite{KLR}, \cite{LL}) systematically studied a large class of connected  self-affine tiles arising from so-called consecutive collinear digit set, i.e., of the form ${\mathcal D}=\{0,1,\dots, c-1\}v$, and their disk-likeness in the plane. They observed an algebraic property of the characteristic polynomial of $A$ to determine the connectedness of $T(A, \mathcal {D})$.  Akiyama and Gjini \cite{AG} also focused on this algebraic property by canonical number systems. On the other hand, Bandt and Wang \cite{BW} and Leung and one of the authors \cite{LLu3} also concerned the disk-like self-affine tiles or the boundary properties by using a technique of neighbor graphs.

Recently, on ${\mathbb R}^2$,  Kirat \cite{Ki} and Deng and Lau \cite{DL}  found out the connected self-affine tiles  $T(A,\mathcal{D})$ among classes of data $(A,\mathcal{D})$ with non-collinear digit sets ${\mathcal D}$ and characterized the disk-like ones. Leung and one of the authors (\cite{LLu1}, \cite{LLu2}) were also interested in  the collinear digit set $\{0,1,m\}v$ and the non-collinear digit set $\{0,v, mAv\}$ with the restriction of  $|\det A|=3$.

In this paper, we  study  more general self-affine sets $T(A,\mathcal{D})$ on ${\mathbb R}^2$ arising from an integral expanding matrix $A$ with characteristic polynomial  $f(x)=x^2+bx+c$ and the consecutive collinear digit set $\mathcal{D}=\{0,1,\dots, m\}v$. We obtain the following main results.

\begin{thm} \label{thm1.1}
Let the characteristic polynomial of $A$ be $f(x)=x^2+bx+c$ and a digit set $\mathcal{D}=\{0,1,\dots, m\}v$
where $m\geq 1$ and $v\in \mathbb{R}^2$ such that $\{v,  Av\}$ are linearly independent. If $\Delta=b^2-4c\geq0$ and
the eigenvalues of $A$ have moduli  $\ge 2$, then

 \medskip
{\rm(i)}  if $c=4$, then $ T(A,\mathcal{D})$ is connected if and only if $m\geq 2$;
 \medskip

{\rm(ii)} otherwise $c\ne 4$, then $ T(A,\mathcal{D})$ is connected if and only if

 \begin{equation*}
m\geq \left\{
\begin{array}{ll}
\max\{c-|b|+1,|b|-1\} & \quad  c>0 \\  \\
|c|-|b|-1 & \quad  c<0.
\end{array}
\right.
\end{equation*}
\end{thm}

\medskip

If $\Delta=b^2-4c<0$, the eigenvalues of $A$ are complex numbers, the self-affine set $T(A, {\mathcal D})$ becomes very complicated.  However, under certain situations, we still obtain some interesting results.

\medskip

\begin{thm} \label{thm1.2}
Let the characteristic polynomial of  $A$ be $f(x)=x^2+bx+c$ and a digit set $\mathcal{D}=\{0,1,\dots, m\}v$
where $m\geq 1$ and $v\in \mathbb{R}^2$ such that $\{v,  Av\}$ are linearly independent.
If  $\Delta=b^2-4c<0$, then $ T(A,\mathcal{D})$ is connected if and only if
\begin{equation*}
m\geq \left\{
\begin{array}{ll}
\max\{c-|b|+1, |b|-1\} & \quad  b^2=3c \\   \\
c-|b|+1 & \quad  b^2=2c, \  b^2=c \\   \\
c-1 & \quad  b=0.
\end{array}
\right.
\end{equation*}
\end{thm}

\medskip

On the other hand,  when the characteristic polynomial of $A$ is of the special form $f(x)=x^2-(p+q)x+pq$ where $|p|,|q|\geq 2$ are integers, and the digit set $\mathcal{D}$ may be  non-consecutively collinear. By letting $f_1(x)=x^2\pm 4x+4$ and  $f_2(x)=x^2\pm7x+12$, we can  characterize the connectedness of the associated  self-affine tile $T(A, {\mathcal D})$ through the following theorem, which is also a generalization of \cite{LLu1}.

\medskip

\begin{thm} \label{thm1.3}
Let the characteristic polynomial of $A$ be $f(x)=x^2-(p+q)x+pq$ and a digit set $\mathcal{D}=\{0,1, \dots, |pq|-2, |pq|-1+s\}v$
where $s\geq 0$, $|p|,|q|\geq 2$ are integers and $v\in \mathbb{R}^2$ such that $\{v,  Av\}$ are linearly independent. Then

 \medskip

{\rm(i)} if  $f\neq f_1, f_2$, then $T(A,\mathcal{D})$ is connected if and only if $s=0$;

 \medskip

{\rm(ii)} if  $f= f_1$ or $f_2$, then $T(A,\mathcal{D})$ is connected if and only if $s=0$ or $1$.
\end{thm}

\medskip

As in the papers previously cited, a lot of calculations are needed in the proofs. But the main methods are algebraic and make full use of the properties of the matrix $A$. We also provide many figures to illustrate our results.

\medskip

The paper is organized as follows: In Section 2, we recall several well-known results on the connectedness of self-affine sets and prove a basic lemma;
Theorems \ref{thm1.1} and \ref{thm1.2} are proved in Section 3, and conclude with an open problem; Theorem \ref{thm1.3} is proved in Section 4.

 \bigskip

\section{\bf Preliminaries}

In the section, we provide several elementary results on self-affine sets $T(A, {\mathcal D})$. We call the digit set $\mathcal {D}$ {\it collinear} if ${\mathcal D}=\{d_1,\dots, d_m\}v$ for some non-zero vector $v\in \mathbb{R}^n$ and $d_1<d_2< \cdots <d_m, \ d_i\in \mathbb{R}$; If $d_{i+1}-d_i=1$, then $\mathcal{D}$ is called a {\it consecutive collinear digit set}.  Let $D = \{d_1, \dots, d_m\}$, $\Delta D=D-D=\{d=d_i-d_j:\ d_i, d_j \in D\}$. Then $\mathcal {D}=Dv$ and $\Delta \mathcal{D} = \Delta D v$. It is easy to see that the connectedness of $T(A, {\mathcal D})$ is invariant under a translation of the digit set, hence we always assume that $d_1=0$ for simplicity.  The following criterion for connectedness of $T(A, {\mathcal D})$ was due to \cite{Ha} or \cite{KL}.

\begin{lemma}\label{lem-con-criterion}
A self-affine set $T(A, {\mathcal D})$ with a consecutive collinear digit set $\mathcal {D}=\{0,1,\dots, m\}v$ is connected if and only if $v\in T-T$.
\end{lemma}

Let ${\mathbb{Z}}[x]$ be the set of polynomials with integer coefficients. A polynomial $f(x)\in {\mathbb{Z}}[x]$ is said to be  {\it
expanding} if all its roots have moduli strictly bigger than $1$. Note that a matrix $A\in M_n({\mathbb{Z}})$ is
expanding if and only if its characteristic polynomial is expanding.
We say that a monic polynomial $f(x)\in {\mathbb{Z}}[x]$ with $|f(0)|=c$ has the {\it Height Reducing Property} (HRP)
if there exists $g(x)\in {\mathbb{Z}}[x]$ such that $$g(x)f(x)=x^k+a_{k-1}x^{k-1}+\cdots+a_1x\pm c$$ where $|a_i|\leq c-1, \  i=1,\dots,c-1$.

\medskip

This property was introduced by Kirat and Lau \cite{KL} to study the connectedness of self-affine tiles  with consecutive collinear digit sets. It was proved that:

\medskip

\begin{pro} \label{thm2.2}
Let $A\in M_n({\mathbb{Z}})$ with $|\det (A)|=c$ be expanding and ${\mathcal{D}}=\{0,1,2,\dots,(c-1)\}v$.  If the characteristic
polynomial of $A$ has the Height Reducing Property, then $T(A, {\mathcal D})$ is connected.
\end{pro}

\medskip

In \cite{KLR}, Kirat {\it et al.} conjectured that all expanding integer monic polynomials have HRP.  Akiyama and Gjini \cite{AG} confirmed it up to $n=4$. But it is still unclear for the higher dimensions. Recently, He {\it et al.} \cite{HKL} developed an algorithm of polynomials about HRP. It may be a good attempt on this problem.

\medskip

Denote the characteristic polynomial of $A$ by $f(x)=x^2+bx+c$, where $b,c \in {\mathbb Z}$. We can regard $A$ as the companion matrix of $f(x)$, i.e.,
$$A=\left[
        \begin{array}{rrrr}
          0  & -c \\
           1 &  -b   \\
    \end{array}
    \right].$$
Let $\Delta=b^2-4c$ be the discriminant. Define $\alpha_i,\beta_i$ by $$A^{-i}v=\alpha_iv+\beta_iAv, \quad i=1,2,\dots.$$
According to the Hamilton-Cayley theorem $f(A)=A^2+bA+cI=0$ where $I$ is a $2\times 2$ identity matrix, the following consequence is well-known (please refer to \cite{LLu1}, \cite{LLu2}).

\medskip

\begin{lemma} \label{thm2.3}
Let $\alpha_i,\beta_i$ be defined as the above. Then $c\alpha_{i+2}+b\alpha_{i+1}+\alpha_i=0$ and $c\beta_{i+2}+b\beta_{i+1}+\beta_i=0$, i.e.,
$$\left[
        \begin{array}{rr}
          \alpha_{i+1} \\
          \alpha_{i+2} \\
    \end{array}
    \right]=\left[
        \begin{array}{rrrr}
          0  & 1 \\
           -1/c &  -b/c   \\
    \end{array}
    \right]^i \left[
        \begin{array}{rr}
          \alpha_1 \\
          \alpha_2 \\
    \end{array}
    \right]; \quad \left[
        \begin{array}{rr}
          \beta_{i+1} \\
          \beta_{i+2} \\
    \end{array}
    \right]=\left[
        \begin{array}{rrrr}
          0  & 1 \\
           -1/c &  -b/c   \\
    \end{array}
    \right]^i \left[
        \begin{array}{rr}
          \beta_1 \\
          \beta_2 \\
    \end{array}
    \right]
$$
and $\alpha_1=-b/c, \alpha_2=(b^2-c)/c^2; \beta_1=-1/b, \beta_2=b/c^2$. Moreover for $\Delta\ne 0$, we have $$\alpha_i=\frac{c(r_1^{i+1}-r_2^{i+1})}{\Delta^{1/2}} \quad\text{and}\quad \beta_i=\frac{-(r_1^i-r_2^i)}{\Delta^{1/2}},$$
where $r_1=\frac{-b+\Delta^{1/2}}{2c}$ and $r_2=\frac{-b-\Delta^{1/2}}{2c}$ are the two roots of $cx^2+bx+1=0$.
\end{lemma}

\medskip

Set
\begin{equation*}
\tilde{\alpha}:= \sum_{i=1}^{\infty}|\alpha_i|, \qquad \tilde{\beta}:= \sum_{i=1}^{\infty}|\beta_i|.
\end{equation*}
Then $\tilde{\alpha}$ and $\tilde{\beta}$ are finite numbers as $r_1, r_2$ have moduli strictly less than $1$.

\medskip

Write $L:=\{\gamma v+\delta Av: \gamma,\delta\in {\mathbb{Z}}\}$, then $L$ is a \emph{lattice} generated by $\{v, Av\}$. For $l\in L\setminus\{0\}$, we call $T+l$  a \emph{neighbor} of $T$ if $T\cap(T+l)\ne \emptyset$.  It is clear that $T+l$ is a neighbor of $T$ if and only if $l\in T-T$, hence $l$ can be expressed as $$l=\sum_{i=1}^{\infty}b_iA^{-i}v \in T-T, ~\text{where}~ b_i\in \Delta D.$$

If $T+l$ is a neighbor of $T$ where $l=\sum_{i=1}^{\infty}b_iA^{-i}v :=\gamma v+\delta Av$, then
\begin{align}\label{estimate}
|\gamma|\leq \max_i|b_i|\tilde \alpha \quad\text{and}\quad |\delta|\leq \max_i|b_i|\tilde \beta.
\end{align}
By multiplying $A$ on both sides of the expression of $l$ and by using $f(A)=0$, it follows that $T-(c\delta+b_1)v+(\gamma - b\delta)Av$ is also a neighbor of $T$. Repeatedly applying this neighbor-generating algorithm, we then can construct a sequence of neighbors: $\{T+l_n\}_{n=0}^{\infty}$, where $l_0=l, l_n=\gamma_n v+\delta_n Av, n\geq 1$ and
\begin{equation}\label{(2.1)}
\left[
        \begin{array}{rr}
          \gamma_n \\
          \delta_n \\
    \end{array}
    \right]=A^n \left[
        \begin{array}{rr}
          \gamma \\
          \delta \\
    \end{array}
    \right]-\sum_{i=1}^n A^{i-1}\left[
        \begin{array}{rr}
          b_{n+1-i} \\
          0\\
    \end{array}
    \right].
\end{equation}
Moreover, $|\gamma_n|\leq \max_i|b_i|\tilde \alpha$ and $|\delta_n|\leq \max_i|b_i|\tilde \beta$ hold for any $n\geq 0$.

\begin{lemma}\label{lem-sym}
If the characteristic polynomial of the expanding matrix $A$ is $x^2+ bx+c$ and that of $B$ is $x^2-bx+c$. Then the self-affine set $T(A, {\mathcal D})$ is connected if and only if $T(B, {\mathcal D})$ is connected where ${\mathcal D}$ is a consecutive collinear digit set.
\end{lemma}

\begin{proof}
Let $B=-A$ and $T_1=T(A, {\mathcal D}), \  T_2=T(-A, {\mathcal D})$. If $l\in T_1-T_1$, then  $$l=\sum_{i=1}^{\infty}b_iA^{-i}v=\sum_{i=1}^\infty b_{2i} (-A)^{-2i}v+ \sum_{i=1}^\infty(-b_{2i-1})(-A)^{-2i+1}v.$$ Thus $l\in T_2-T_2$, and vice versa.
\end{proof}

\medskip

To get the necessary conditions of Theorems \ref{thm1.1} , \ref{thm1.2} and \ref{thm1.3} , we need the exact values of $\tilde{\alpha}$ and $\tilde{\beta}$.

\medskip

\begin{lemma}\label{thm2.5}
Let the characteristic polynomial of the expanding matrix $A$ be $f(x)=x^2+bx+c$, where $b, c$ are integers and $\Delta=b^2-4c\geq0$. Then
\begin{equation*}
\tilde{\alpha}= \left\{
\begin{array}{ll}
\frac{|b|-1}{c-|b|+1} &\quad  c>0 \\ \\
\frac{|b|+1}{|c|-|b|-1} & \quad  c<0;
\end{array}
\right. \ \ \ \ \ \ \ \ \ \  \ \ \ \ \
\tilde{\beta}= \left\{
\begin{array}{ll}
\frac{1}{c-|b|+1} &\quad c>0 \\ \\
\frac{1}{|c|-|b|-1} &  \quad  c<0.
\end{array}
\right.
\end{equation*}
\end{lemma}

\medskip

\begin{proof}
 Let $x_1,x_2$ denote the roots of $x^2+bx+c=0$.

\noindent (1) \ $c>0$. \quad   If $|x_1|>|x_2|$, then  $\displaystyle |\alpha_i|=\frac{c}{|x_1-x_2|}(\frac{1}{|x_2|^{i+1}}-\frac{1}{|x_1|^{i+1}})$,
 $\displaystyle |\beta_i|=\frac{1}{|x_1-x_2|}(\frac{1}{|x_2|^{i}}-\frac{1}{|x_1|^{i}})$. Hence
\begin{eqnarray*}
\tilde{\alpha}=\sum_{i=1}^\infty|\alpha_i|=\sum_{i=1}^\infty\frac{c}{|x_1-x_2|}(\frac{1}{|x_2|^{i+1}}-\frac{1}{|x_1|^{i+1}})=
\frac{|x_1+x_2|-1}{(|x_1|-1)(|x_2|-1)}=\frac{|b|-1}{c-|b|+1};
\end{eqnarray*}
\begin{eqnarray*}
\tilde{\beta}=\sum_{i=1}^\infty|\beta_i|=\sum_{i=1}^\infty\frac{1}{|x_1-x_2|}(\frac{1}{|x_2|^{i}}-\frac{1}{|x_1|^{i}})=
\frac{1}{(|x_1|-1)(|x_2|-1)}=\frac{1}{c-|b|+1}.
\end{eqnarray*}
Similarly for $\displaystyle |x_2|>|x_1|$.

If $\displaystyle |x_1|=|x_2|=|b|/2$,  by Lemma \ref{thm2.3} and  a simple calculation,
it follows that  $\displaystyle |\alpha_i|=\frac{i+1}{|x_1|^i}$ and $\displaystyle |\beta_i|=\frac{i}{|x_1|^{i+1}}$. Thus
\begin{eqnarray*}
(1-\frac{1}{|x_1|})\tilde{\alpha}&=&\sum_{i=1}^\infty\frac{i+1}{|x_1|^i}-\sum_{j=1}^\infty\frac{j+1}{|x_1|^{j+1}}
=\frac{2}{|x_1|}+\sum_{\ell=1}^\infty\frac{1}{|x_1|^{\ell+1}}\\ \nonumber
&=& \frac{2}{|x_1|}+\frac{1}{|x_1|(|x_1|-1)}=\frac{2|x_1|-1}{|x_1|(|x_1|-1)};
\end{eqnarray*}
\begin{eqnarray*}
(|x_1|-1)\tilde{\beta}&=&\sum_{i=1}^\infty\frac{i}{|x_1|^i}-\sum_{j=1}^\infty\frac{j}{|x_1|^{j+1}}=\frac{1}{|x_1|}+\sum_{\ell=1}^\infty\frac{1}{|x_1|^{\ell+1}}\\  \nonumber
&=& \frac{1}{|x_1|}+\frac{1}{|x_1|(|x_1|-1)}=\frac{1}{(|x_1|-1)}
\end{eqnarray*}
which implies $\displaystyle\tilde{\alpha}=\frac{2|x_1|-1}{(|x_1|-1)^2}=\frac{|b|-1}{c-|b|+1}$ and $\displaystyle\tilde{\beta}=\frac{1}{(|x_1|-1)^2}=\frac{1}{c-|b|+1}$.

\medskip

\noindent (2) \ $c<0$. \quad  Without loss of generality, we can assume  $|x_1|\geq |x_2|$, then
\begin{eqnarray*}
\tilde{\alpha}&=&\sum_{i=1}^\infty|\alpha_i|
=\frac{|c|}{|x_1|+|x_2|}\left(\sum_{i=1}^\infty(\frac{1}{|x_1|^{2i+1}}+\frac{1}{|x_2|^{2i+1}})+\sum_{i=1}^\infty(\frac{1}{|x_2|^{2i}}-\frac{1}{|x_1|^{2i}})\right)\\
&=& \frac{|c|}{|x_1|+|x_2|}\left(\frac{1}{|x_1|(|x_1|^2-1)}+\frac{1}{|x_2|(|x_2|^2-1)}+\frac{1}{|x_2|^2-1}-\frac{1}{|x_1|^2-1}\right)\\
&=&\frac{|c|}{|x_1|+|x_2|}\left(\frac{1}{|x_2|(|x_2|-1)}-\frac{1}{|x_1|(|x_1|+1)}\right) \\
&=&\frac{|x_1|-|x_2|+1}{(|x_1|+1)(|x_2|-1)}\\
&=&\frac{|b|+1}{|c|-|b|-1};
\end{eqnarray*}
\begin{eqnarray*}
\tilde{\beta}&=&\sum_{i=1}^\infty|\beta_i|=
\frac{1}{|x_1|+|x_2|}\left(\sum_{i=1}^\infty(\frac{1}{|x_1|^{2i-1}}+\frac{1}{|x_2|^{2i-1}})+\sum_{i=1}^\infty(\frac{1}{|x_2|^{2i}}-\frac{1}{|x_1|^{2i}})\right)\\
&=& \frac{1}{|x_1|+|x_2|}\left(\frac{|x_1|}{|x_1|^2-1}+\frac{|x_2|}{|x_2|^2-1}+\frac{1}{|x_2|^2-1}-\frac{1}{|x_1|^2-1}\right)\\
&=&\frac{1}{|x_1|+|x_2|}\left(\frac{1}{|x_2|-1}+\frac{1}{|x_1|+1}\right)\\
&=& \frac{1}{|c|-|b|-1}.
\end{eqnarray*}
\end{proof}

\bigskip

\section{\bf Consecutive collinear digit set}

In the section, we characterize the connectedness of the self-affine sets $T(A,\mathcal{D})$ associated with digit sets $\mathcal{D}=\{0,1,\dots, m\}v$. The necessary and sufficient conditions are  given for $ T(A,\mathcal{D})$ to be connected.

\begin{thm} \label{thm3.1}
Let the characteristic polynomial of the expanding integer matrix $A$ be $f(x)=x^2+bx+c$ and a digit set $\mathcal{D}=\{0,1,\dots, m\}v$
where $m\geq 1$ is  integral and $v\in \mathbb{R}^2$ such that $\{v,  Av\}$ are linearly independent. If $\Delta=b^2-4c\geq0$ and
the eigenvalues of $A$ have moduli   $\ge 2$, then

 \medskip
{\rm(i)}  if $c=4$, then $ T(A,\mathcal{D})$ is connected if and only if $m\geq 2$;
 \medskip

{\rm(ii)} otherwise $c\ne 4$, then $ T(A,\mathcal{D})$ is connected if and only if

 \begin{equation*}
m\geq \left\{
\begin{array}{ll}
\max\{c-|b|+1,|b|-1\} & \quad  c>0 \\  \\
|c|-|b|-1 & \quad  c<0.
\end{array}
\right.
\end{equation*}
\end{thm}

 \medskip

\begin{proof}
Let $T+l$ be a neighbor of $T$, then $l=\gamma v+\delta Av=\sum_{i=1}^{\infty}b_iA^{-i}v$, $b_i\in \Delta D=\{0,\pm 1,\pm2,\dots, \pm m\}$.
By (\ref{estimate}), we have
\begin{align}\label{3.2}
|\gamma|\leq m\tilde{\alpha}; \qquad   |\delta|\leq m\tilde{\beta}.
\end{align}
Suppose $(T+\ell_1v)\cap (T+\ell_2v)\ne\emptyset$ for $0\leq \ell_1<\ell_2\leq m$,
 then $(\ell_2-\ell_1)v=\sum_{i=1}^{\infty}b_iA^{-i}v, \ b_i\in \Delta D$.
 By \eqref{(2.1)}, we obtain  $l^*= -((\ell_2-\ell_1)c+b_2)v-(b(\ell_2-\ell_1)+b_1)Av$.

\medskip

\noindent (i)  If $c=4$, then $|b|=4$,  by \eqref{3.2} and Lemma \ref{thm2.5}, the connectedness of $T(A, {\mathcal D})$ can imply that

\begin{align*}
(\ell_2-\ell_1)|b|-m\leq |(\ell_2-\ell_1)b+b_1|\leq \frac{m}{c-|b|+1}= m,
\end{align*}
which further implies  that  $m\ge 2(\ell_2-\ell_1)\ge 2$.

\medskip

Conversely, if $b=-4$, then $f(x)=x^2-4x+4$. By using $f(A)=0$ and $Af(A)=0$, we have $A^3-3A^2+4I=0$ and
$A^3-A^2=2A(A-I)+2(A-I)-2I$ which yields
\begin{equation*}
(A-I)=2A^{-1}(A-I)+2A^{-2}(A-I)-2A^{-2}
\end{equation*} and
\begin{equation*}
I=2A^{-1}+2A^{-2}-2\sum_{i=3}^{\infty}{A}^{-i}.
\end{equation*}
Then \begin{equation*}
v=2A^{-1}v+2A^{-2}v-2\sum_{i=3}^{\infty}{A}^{-i}v\in T-T.
\end{equation*}
Therefore $T$ is connected by Lemma \ref{lem-con-criterion}. (see Figure \ref{fig1})

If $b=4$, then $f(x)=x^2+4x+4$. By Lemma \ref{lem-sym}, the connectedness of $T(A, {\mathcal D})$ is the same as that of
$T(-A, {\mathcal D})$ in which the characteristic polynomial of $-A$ is $f(x)=x^2-4x+4$.

\medskip

\noindent (ii)  Necessity:\quad  If $c>0$, then by \eqref{3.2} and Lemma \ref{thm2.5} , we have
\begin{align}\label{3.3}
(\ell_2-\ell_1)c-m\leq |(\ell_2-\ell_1)c+b_2|\leq m\frac{|b|-1}{c-|b|+1};
\end{align}
\begin{align}\label{3.4}
(\ell_2-\ell_1)|b|-m\leq |(\ell_2-\ell_1)b+b_1|\leq \frac{m}{c-|b|+1}.
\end{align}
\eqref{3.3} implies  that $m\geq (\ell_2-\ell_1)(c-|b|+1)\geq c-|b|+1$.

Now we prove  $m\geq|b|-1$. Let $x_1,x_2$ denote the roots of $x^2+bx+c=0$, if $m<|b|-1$, then $m\leq|b|-2=|x_1|-1+|x_2|-1$.
We have
\begin{align}\label{3.5}
\frac{m}{(|x_1|-1)(|x_2|-1)}\leq \frac{1}{|x_1|-1}+\frac{1}{|x_2|-1}<2.
\end{align}
The last strict inequality holds due to the fact that $c>4$ and $|x_1|,|x_2|\geq 2$. From \eqref{3.4} and \eqref{3.5}, it follows that
\begin{align}\label{3.6}
2\leq|b|-m\leq(l_2-l_1)|b|-m \leq \frac{m}{c-|b|+1}=\frac{m}{(|x_1|-1)(|x_2|-1)}<2
\end{align} which is a contradiction.
Hence  $m\geq|b|-1$.

\medskip

If $c<0$, then by \eqref{3.2} and Lemma \ref{thm2.5}, we have
\begin{align}\label{3.7}
(\ell_2-\ell_1)|c|-m\leq |(\ell_2-\ell_1)c+b_2|\leq m\frac{|b|+1}{|c|-|b|-1}
\end{align}
implying  that $m\geq (l_2-l_1)(|c|-|b|-1)\geq |c|-|b|-1$.

\bigskip

\noindent Sufficiency: \quad If $c>0$, and $m\geq \max\{c-|b|+1, |b|-1\}$, it suffices to show that $v\in T-T$ by Lemma \ref{lem-con-criterion}.

\medskip
When $b<0$, by using $f(A)=A^2+bA+cI=0$, we have $A^2+bA-(b+1)I=-(c+b+1)I$, i.e., $(A-I)(A+(b+1)I)=-(c+b+1)I$. It follows that
$$I+(b+1)A^{-1}=-(c+b+1)A^{-1}\sum_{i=1}^{\infty}A^{-i}.$$
Hence
\begin{eqnarray*}
I&=&(-b-1)A^{-1}+\sum_{i=2}^{\infty}-(c+b+1)A^{-i}\\
 &=&(|b|-1)A^{-1}+\sum_{i=2}^{\infty}-(c-|b|+1)A^{-i}.
\end{eqnarray*}
Then $v=(|b|-1)A^{-1}v+\sum_{i=2}^{\infty}-(|c|-|b|+1)A^{-i}v\in T-T$, and $T$ is connected. When $b>0$,
Lemma \ref{lem-sym} and the above argument also yield that $T$ is connected. (see Figure \ref{fig2})

\medskip

If $c<0$, suppose $m\geq |c|-|b|-1$, then $m\geq|b|+1$ (Indeed, if $x_1,x_2$ is the roots of $x^2+bx+c=0$, without loss of generality, we let  $|x_1|\geq|x_2|$,
then$|c|-|b|-1=|x_1x_2|-(|x_1|-|x_2|)-1=(|x_1|+1)(|x_2|-1)\geq |x_1|+1>|x_1|-|x_2|+1=|b|+1$).

\medskip

When $b<0$, by using $f(A)=A^2+bA+cI=0$, we have $A^2+A=(-b+1)(A+I)+(-c+b-1)I$. Then
\begin{eqnarray*}
I&=&(-b+1)A^{-1}+(-c+b-1)\sum_{i=2}^{\infty}{(-A)}^{-i}\\
&=&(|b|+1)A^{-1}+(|c|-|b|-1)\sum_{i=2}^{\infty}{(-A)}^{-i}\\
&=&(|b|+1)A^{-1}+(|c|-|b|-1)\sum_{k=1}^{\infty}{A}^{-2k} -(|c|-|b|-1)\sum_{k=1}^{\infty}{A}^{-2k-1}.
\end{eqnarray*}
Hence $v\in T-T$ and $T$ is connected by Lemma \ref{lem-con-criterion}.  When $b>0$, Lemma \ref{lem-sym}
and the above argument also yield that $T$ is connected. (see Figure \ref{fig3})
\end{proof}

\begin{figure}[h]
 \centering
\subfigure[m=1]{
 \includegraphics[width=5cm]{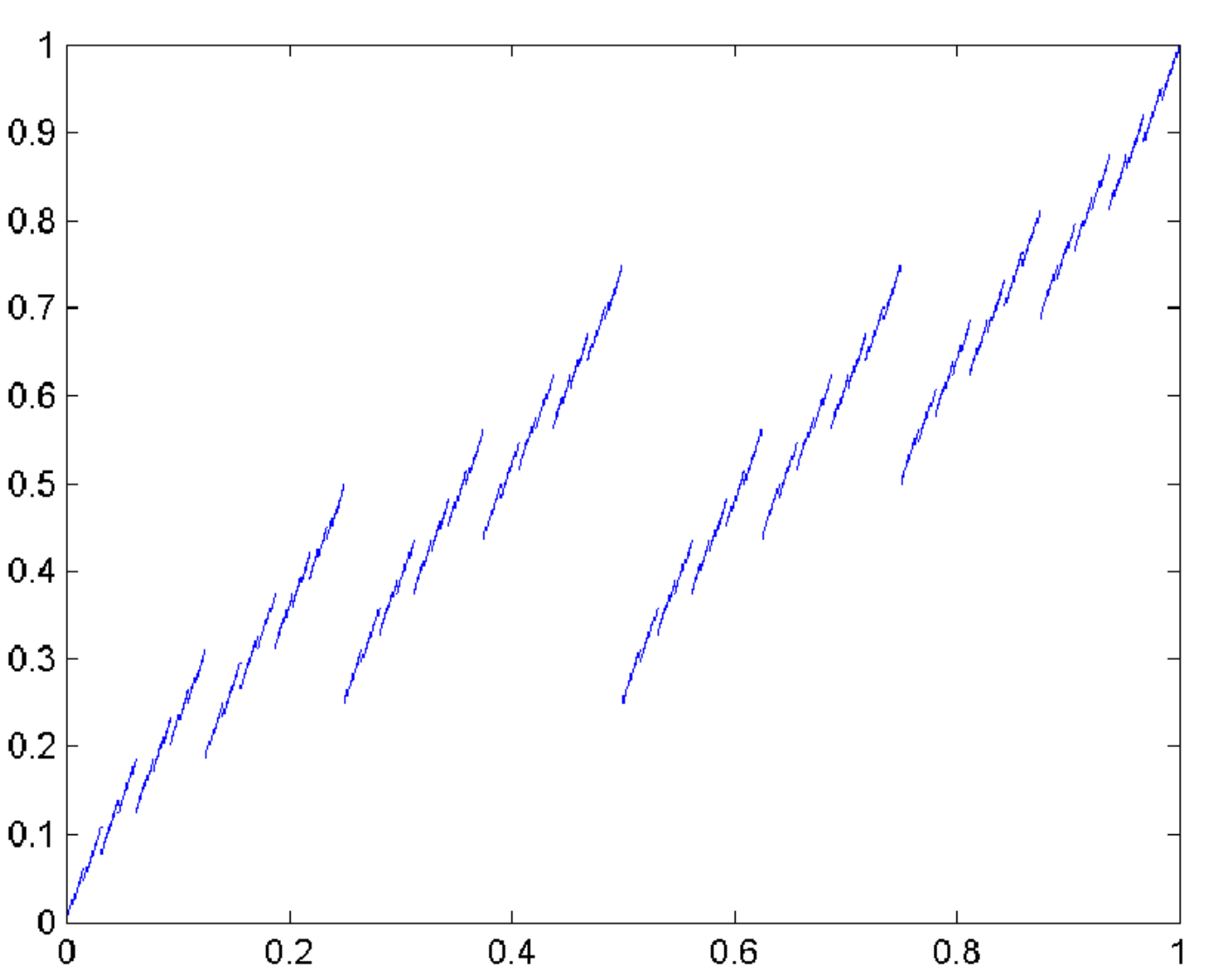}
}\qquad
 \subfigure[m=2]{
\includegraphics[width=5cm]{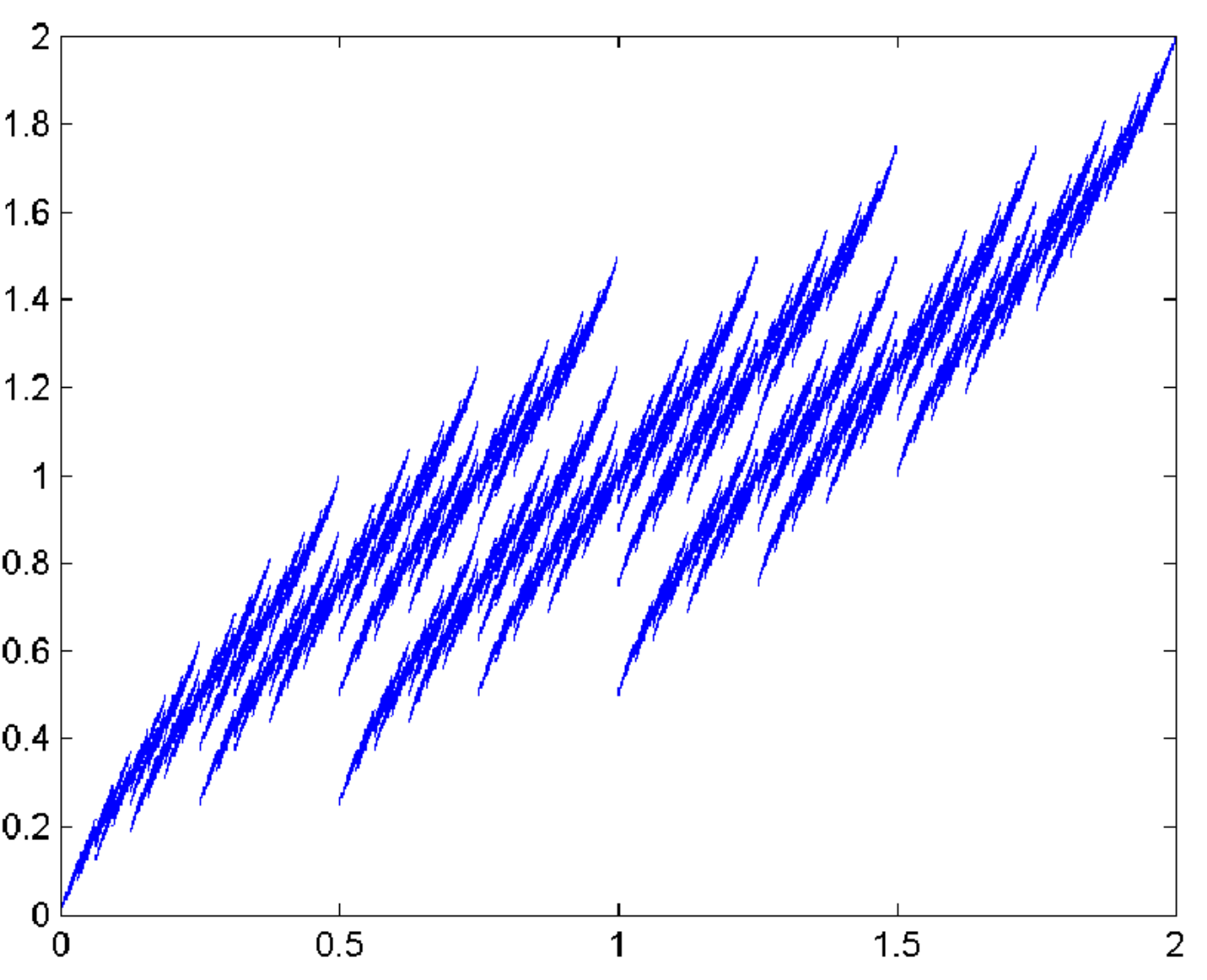}
}
\caption{(a) is disconnected and (b) is connected where $A=[2,0;-1,2], v=(1,0)^t$.}\label{fig1}
\end{figure}

\begin{figure}[h]
\centering
\subfigure[m=13]{
 \includegraphics[width=5cm]{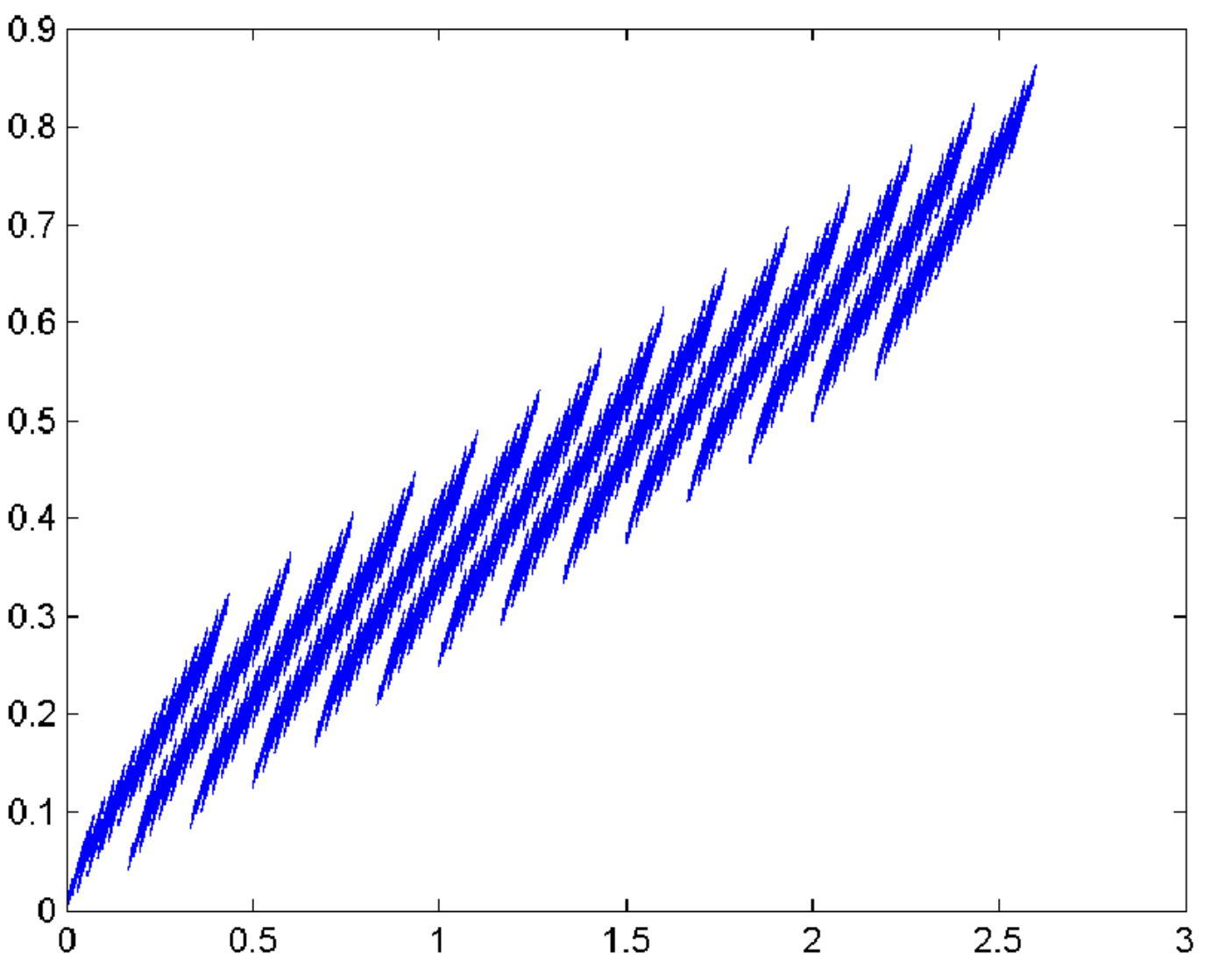}
 }\qquad
\subfigure[m=15]{
\includegraphics[width=5cm]{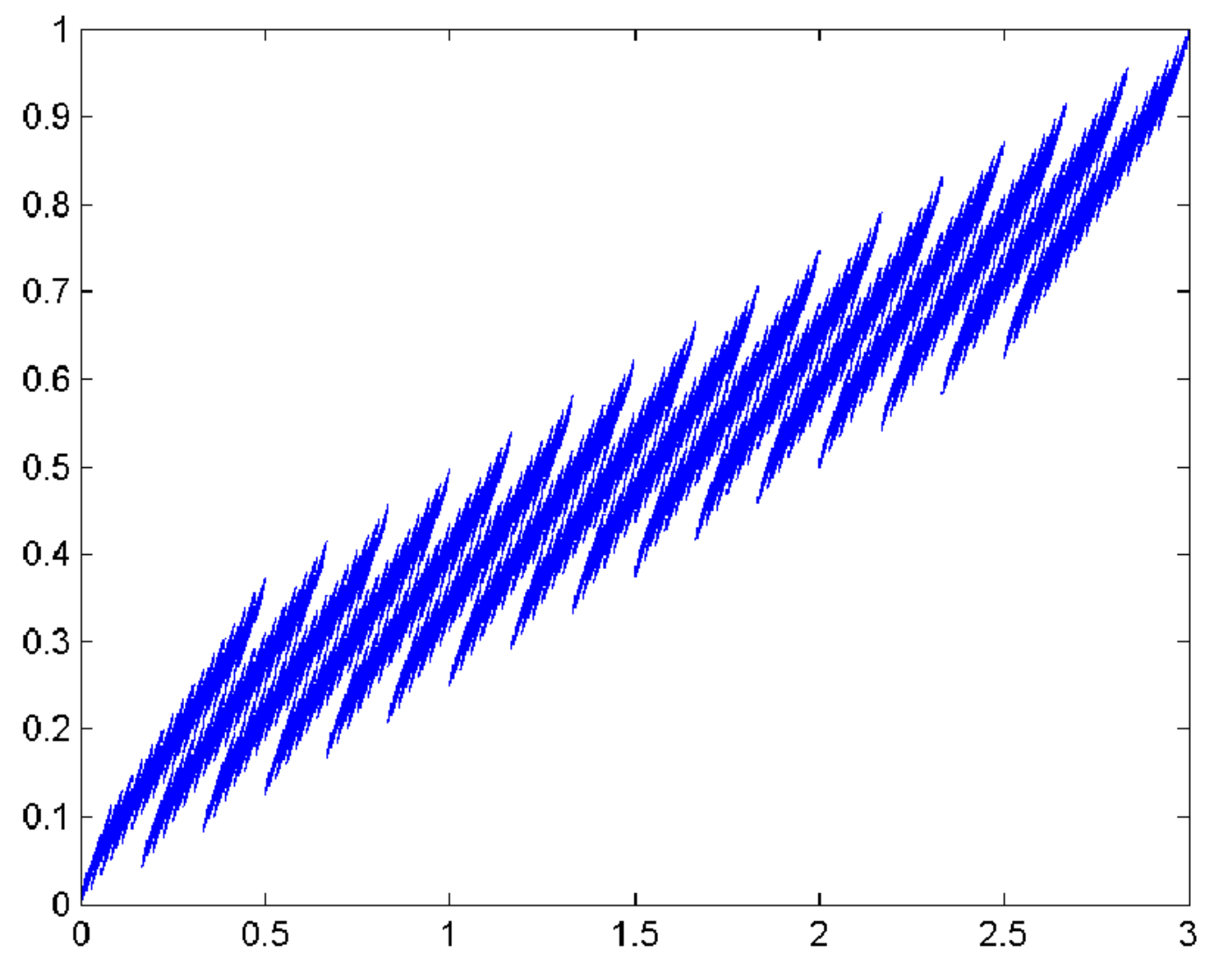}
}
\caption{(a) is disconnected and (b) is connected where $A=[6,0;-1,4]$, $v=(1,0)^t$.}\label{fig2}
\end{figure}

\begin{figure}[h]
 \centering
\subfigure[m=19]{
 \includegraphics[width=5cm]{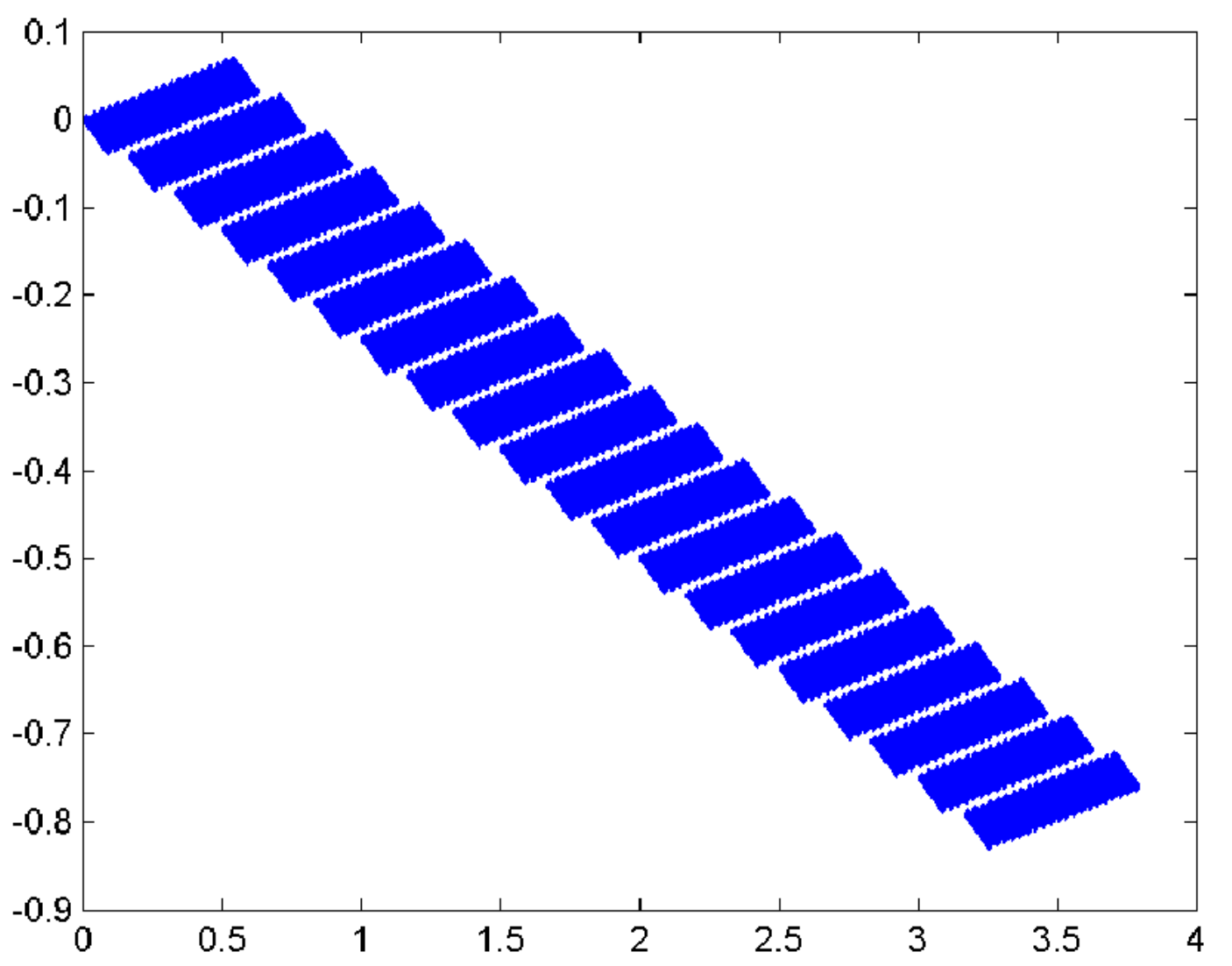}
}\qquad
\subfigure[m=21]{
 \includegraphics[width=5cm]{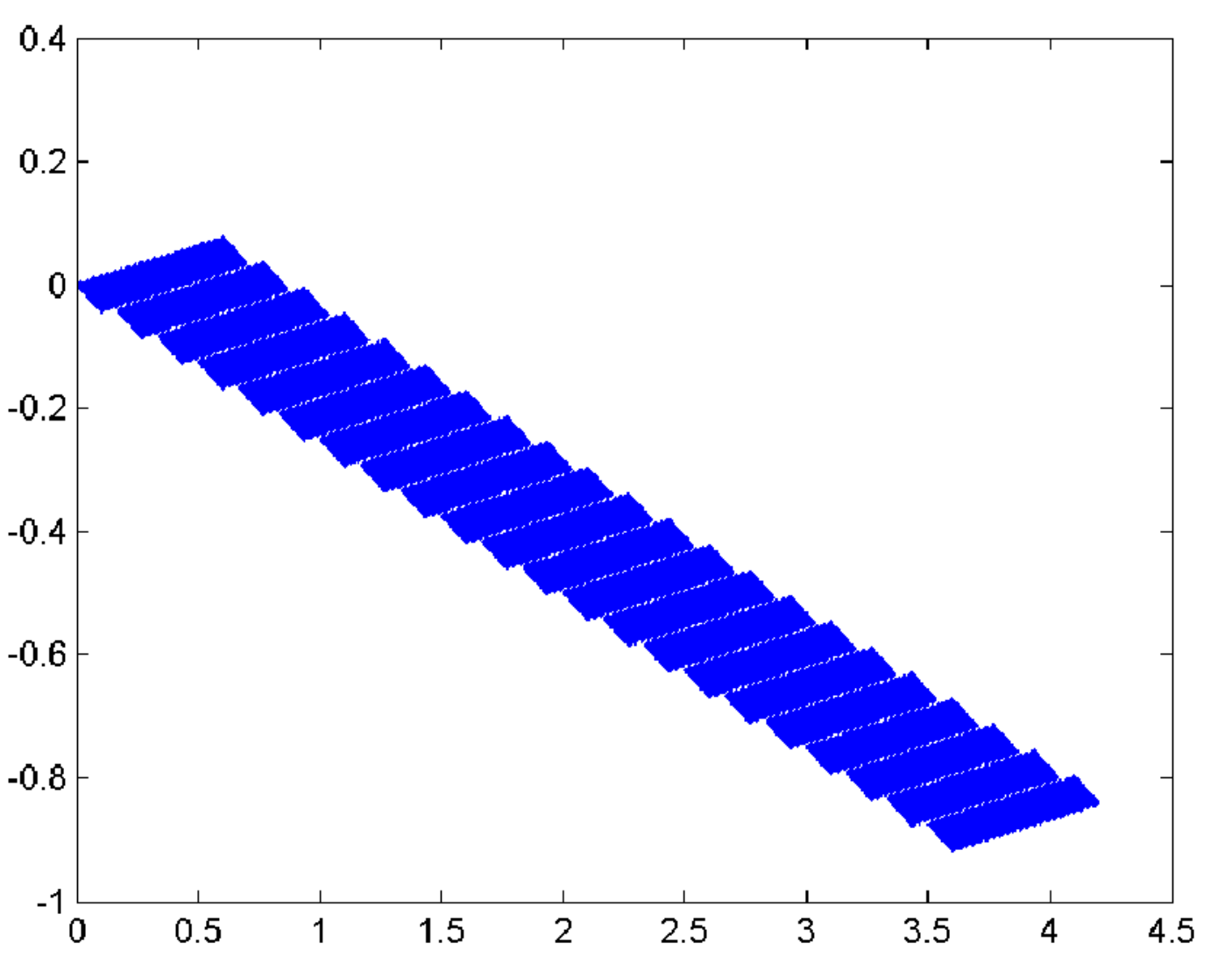}
}
\caption{(a) is disconnected and (b) is connected where $A=[6,0;-1,-4]$, $v=(1,0)^t$}\label{fig3}
\end{figure}

In the proof above, the condition that eigenvalues of $A$ have moduli $\ge 2$ is essential. If otherwise, in the case that the moduli of the eigenvalues $< 2$,
e.g., the moduli are close to $1$, we have no idea about the conditions for $ T(A,\mathcal{D})$ to be connected by estimating $\tilde{\alpha}$  or $\tilde{\beta}$.

 \medskip

On the other hand, if  $\Delta=b^2-4c<0$, it is also difficult to  compute the exact values of  $\tilde{\alpha}$  and  $\tilde{\beta}$ in general. However,  under certain special situations, the exact values of $\tilde{\alpha}$ and  $\tilde{\beta}$ can still be calculated as well.

\begin{thm} \label{thm3.2}
Let the characteristic polynomial of $A$ be $f(x)=x^2+bx+c$ and a digit set $\mathcal{D}=\{0,1,\dots, m\}v$ where $m\geq 1$ and $v\in \mathbb{R}^2$ such that $\{v,  Av\}$ are linearly independent.  If  $\Delta=b^2-4c<0$, then $ T(A,\mathcal{D})$ is connected if and only if
\begin{equation}\label{4.3}
m\geq \left\{
\begin{array}{ll}
\max\{c-|b|+1, |b|-1\} & \quad  b^2=3c \\  \\
c-|b|+1 & \quad  b^2=2c, \  b^2=c \\  \\
c-1 & \quad  b=0.
\end{array}
\right.
\end{equation}
\end{thm}

\begin{proof}
From Lemma \ref{lem-sym}, we can suppose $b<0$. Let  $r_{1,2}=\frac{-b\pm\sqrt{b^2-4c}}{2c}$ be the complex roots of $cx^2+bx+1=0$ as in Lemma \ref{thm2.3} and let  $r:=|r_1|=|r_2|=\frac{1}{\sqrt{c}}$. Then  $r_{1}=r e^{i\theta}$ and $r_{2}=r e^{-i\theta}$ where $\theta$ is the argument of $r_{1}$. We show the necessity first by assuming $T(A, {\mathcal D})$ is connected.

\medskip

 (i) If $b^2=3c$, then $\theta=\frac{\pi}{6}$ and
\begin{equation*}
 |\alpha_i|=\left|\frac{c(r_1^{i+1}-r_2^{i+1})}{\Delta^{1/2}}\right|=\frac{2cr^{i+1}|\sin((i+1)\frac{\pi}{6})|}{\sqrt{c}};
\end{equation*}

\begin{equation*}
 |\beta_i|=\left|\frac{(r_1^{i+1}-r_2^{i+1})}{\Delta^{1/2}}\right|=\frac{2r^{i+1}|\sin((i+1)\frac{\pi}{6})|}{\sqrt{c}}.
\end{equation*}

Hence

\begin{eqnarray*}
\tilde{\alpha}&=&\sum_{i=1}^\infty|\alpha_i|=\sqrt{c}\left((\sqrt{3}r^2+2r^3+\sqrt{3}r^4+r^5+r^7)\sum_{j=0}^\infty r^{6j} \right)\\
&=&\frac{3|b|^5+6b^4+9|b|^3+9b^2+27}{b^6-27};
\end{eqnarray*}
Analogous to \eqref{3.3}, we have
\begin{align*}
c-m\le (\ell_2-\ell_1)c-m\leq |(\ell_2-\ell_1)c+b_2|\leq m\frac{3|b|^5+6b^4+9|b|^3+9b^2+27}{b^6-27}
\end{align*}
which implies that
\begin{eqnarray*}
m&\geq& \frac{b^6-27}{3(b^4+3|b|^3+6b^2+9|b|+9)} \\
&=& \frac{b^2}{3}-|b|+\frac{3b^4+9|b|^3+18b^2+27|b|-27}{3b^4+9|b|^3+18b^2+27|b|+27}\\
&=& c-|b|+\frac{3b^4+9|b|^3+18b^2+27|b|-27}{3b^4+9|b|^3+18b^2+27|b|+27}.
\end{eqnarray*}
Thus $m\geq c-|b|+1$ as $0<\frac{3b^4+9|b|^3+18b^2+27|b|-27}{3b^4+9|b|^3+18b^2+27|b|+27}<1$ and $m$ is integral. (see Figure \ref{fig7})

If $|b|>3$, then $c-|b|+1>|b|-1$ is always true; if $|b|=3$ then $c=3$, and
\begin{eqnarray*}
\tilde{\beta}=\sum_{i=1}^\infty|\beta_i|=\frac{1}{\sqrt{c}}\left((r+\sqrt{3}r^2+2r^3+\sqrt{3}r^4+r^5)\sum_{j=0}^\infty r^{6j} \right)
=\frac{14}{13}.
\end{eqnarray*}
Analogous to \eqref{3.4}, we have
\begin{align*}
3-m\le (\ell_2-\ell_1)3-m\leq |(\ell_2-\ell_1)b+b_1|\leq m\frac{14}{13}
\end{align*}
and $m\geq \frac{39}{27}$.  Therefore $m\geq 2=|b|-1$.

\medskip

(ii) If $b^2=2c$, then $\theta=\frac{\pi}{4}$ and
\begin{equation*}
 |\alpha_i|=\left|\frac{c(r_1^{i+1}-r_2^{i+1})}{\Delta^{1/2}}\right|=\frac{2cr^{i+1}|\sin((i+1)\frac{\pi}{4})|}{|b|}.
\end{equation*}

Then

\begin{eqnarray*}
\tilde{\alpha}=\sum_{i=1}^\infty|\alpha_i|=|b|\left(r^2\sum_{j=0}^\infty r^{4j}+\frac{\sqrt{2}}{2}r^3\sum_{k=0}^\infty r^{2k} \right)=
\frac{2|b|^3+2b^2+4}{b^4-4}.
\end{eqnarray*}
Analogous to  \eqref{3.3}, we have
\begin{align*}
c-m\le (\ell_2-\ell_1)c-m\leq |(\ell_2-\ell_1)c+b_2|\leq m\frac{2|b|^3+2b^2+4}{b^4-4}
\end{align*}
implying   that
\begin{eqnarray*}
m&\geq& \frac{b^4-4}{2(b^2+2|b|+2)}\\
&=& \frac{b^2}{2}-|b|+\frac{2b^2+4|b|-4}{2b^2+4|b|+4}\\
&=& c-|b|+\frac{2b^2+4|b|-4}{2b^2+4|b|+4}.
\end{eqnarray*}
Hence $m\geq c-|b|+1$ as $0<\frac{2b^2+4|b|-4}{2b^2+4|b|+4}<1$ and $m$ is integral. (see Figure \ref{fig8})

\medskip

(iii) If $b^2=c$, then $\theta=\frac{\pi}{3}$. By the similar discussion of (i) above, it follows that
\begin{eqnarray*}
\tilde{\alpha}=\sum_{i=1}^\infty|\alpha_i|=\frac{b^2+1}{|b|^3-1}
\end{eqnarray*}
and
\begin{eqnarray*}
m\geq \frac{|b|^3-1}{|b|+1}=b^2-|b|+\frac{|b|-1}{|b|+1}
=c-|b|+\frac{|b|-1}{|b|+1}.
\end{eqnarray*}
Hence $m\geq c-|b|+1$ as $0<\frac{|b|-1}{|b|+1}<1$. (see Figure \ref{fig9})

\medskip

(iv) If $b=0$, then $\theta=\frac{\pi}{2}$. Similarly, we have
\begin{eqnarray*}
\tilde{\alpha}=\sum_{i=1}^\infty|\alpha_i|=\frac{c}{\sqrt{c}}\sum_{i=1}^\infty r^{2i+1}=\frac{1}{c-1}
\end{eqnarray*}
and then $m\geq c-1$.

On the contrary, for the sufficiency, if $m$ satisfies \eqref{4.3}, then $m\geq |b|-1$  for cases (ii) and (iii). With the similar proof as in  Theorem \ref{thm3.1}, we can conclude that $v\in T-T$ and $T(A,\mathcal{D})$ is  connected. The connectedness of  case
(iv) comes from the Proposition \ref{thm2.2} directly.
\end{proof}

\begin{figure}[h]
 \centering
\subfigure[m=6]{
 \includegraphics[width=5cm]{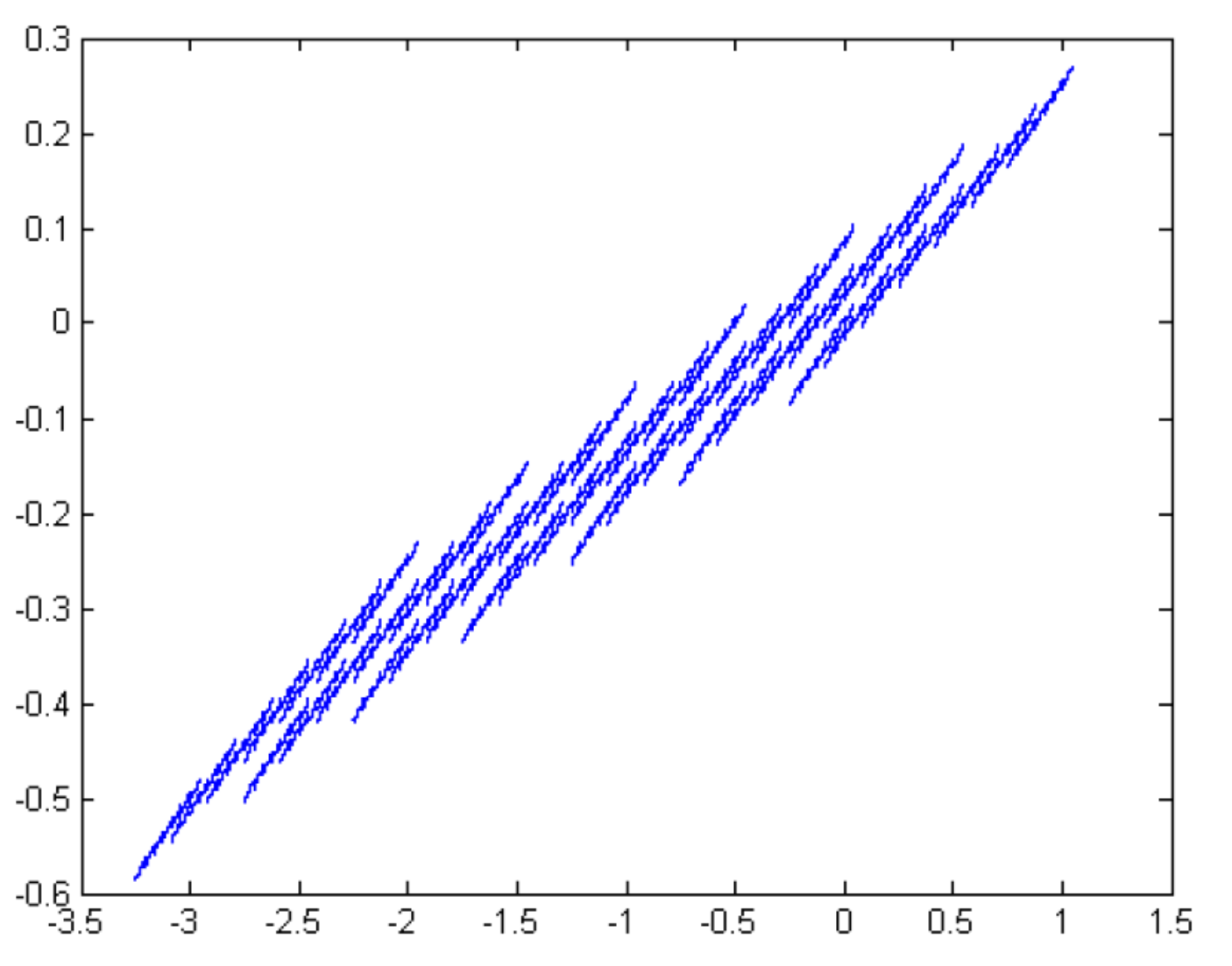}
 }\qquad
  \subfigure[m=7]{
 \includegraphics[width=5cm]{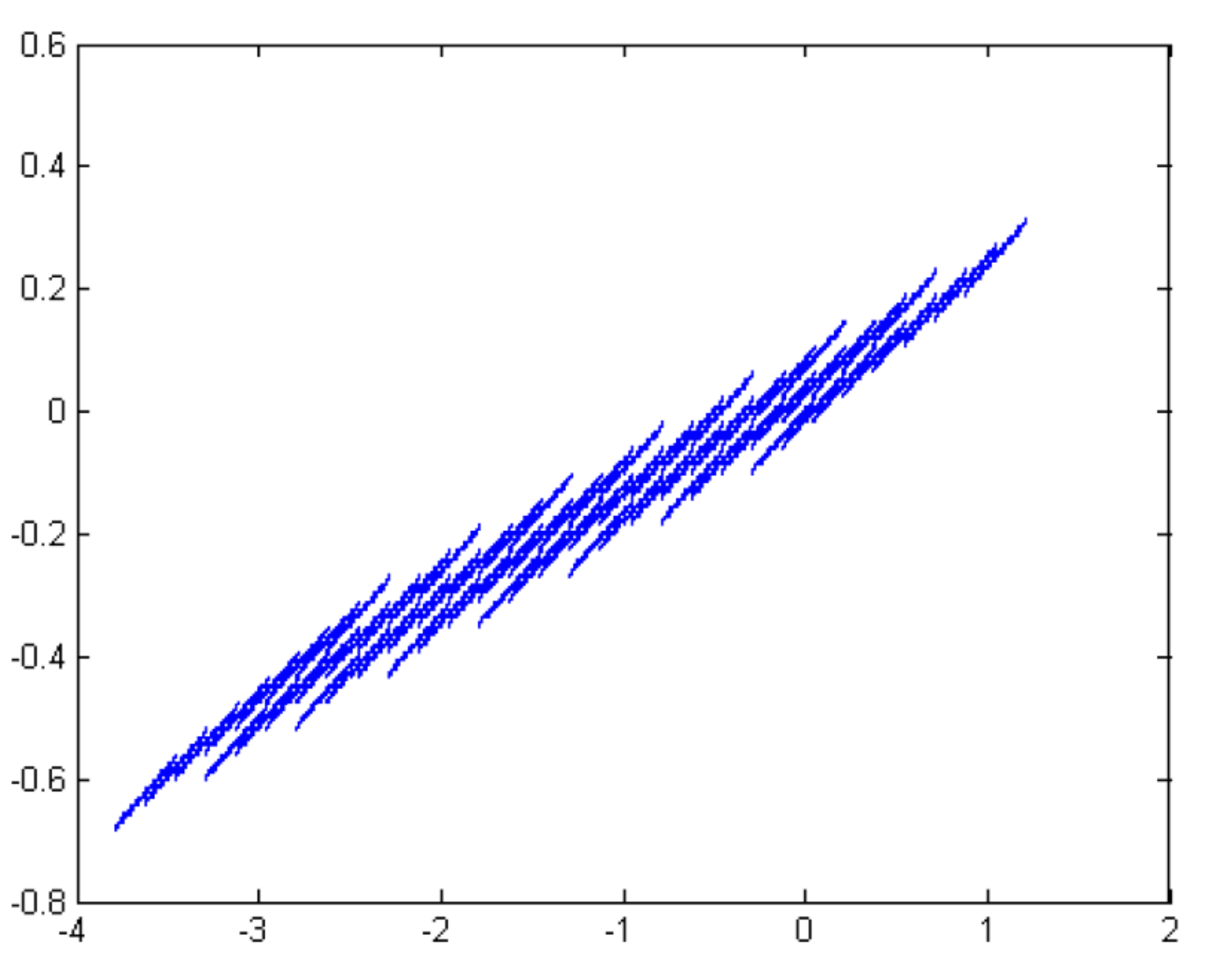}
}\caption{(a) is disconnected and (b) is connected where $A=[0,-12; 1,-6], v=(1,0)^t$.}\label{fig7}
\end{figure}

\begin{figure}[h]
 \centering
\subfigure[m=4]{
 \includegraphics[width=5cm]{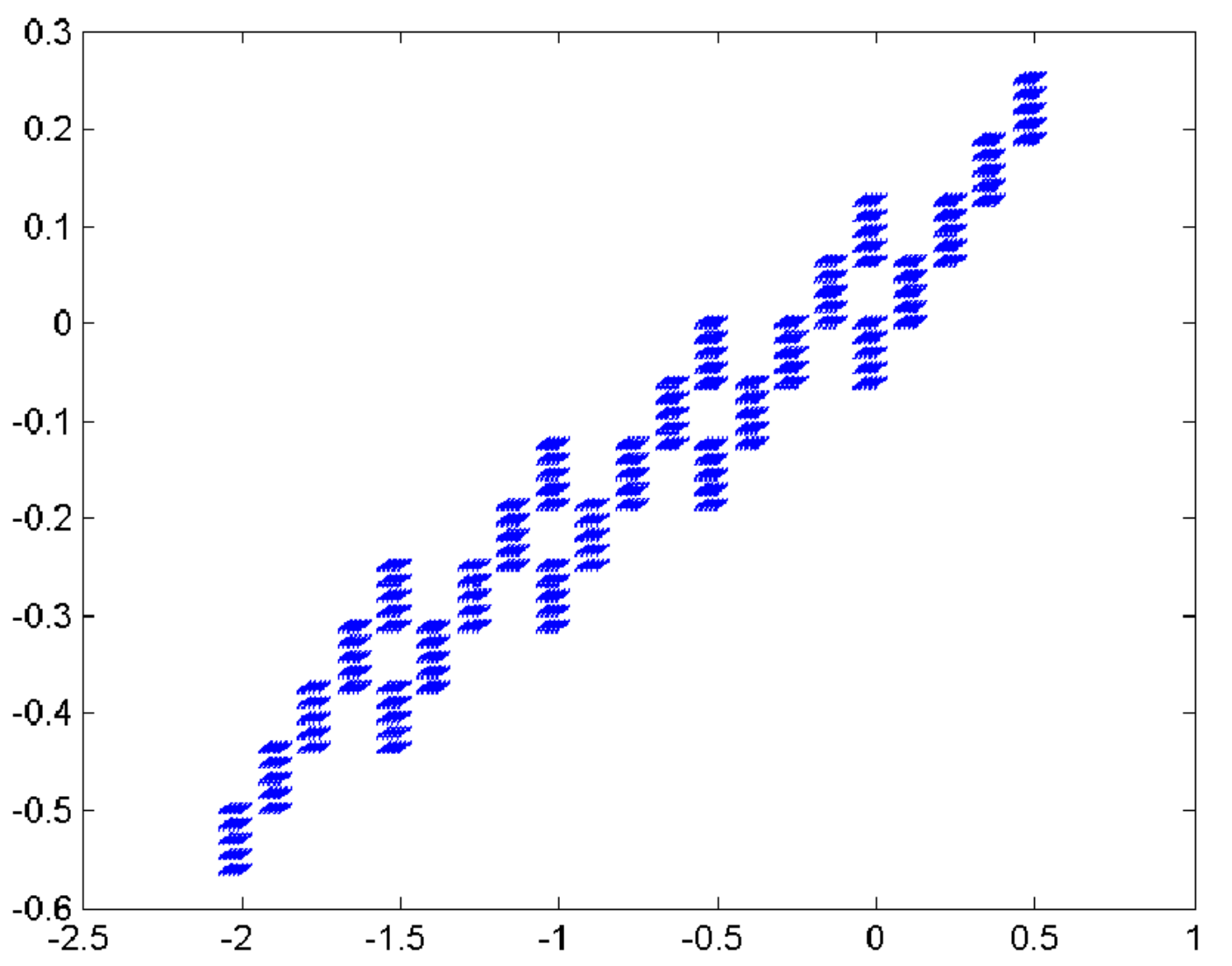}
 }\qquad
 \subfigure[m=5]{
 \includegraphics[width=5cm]{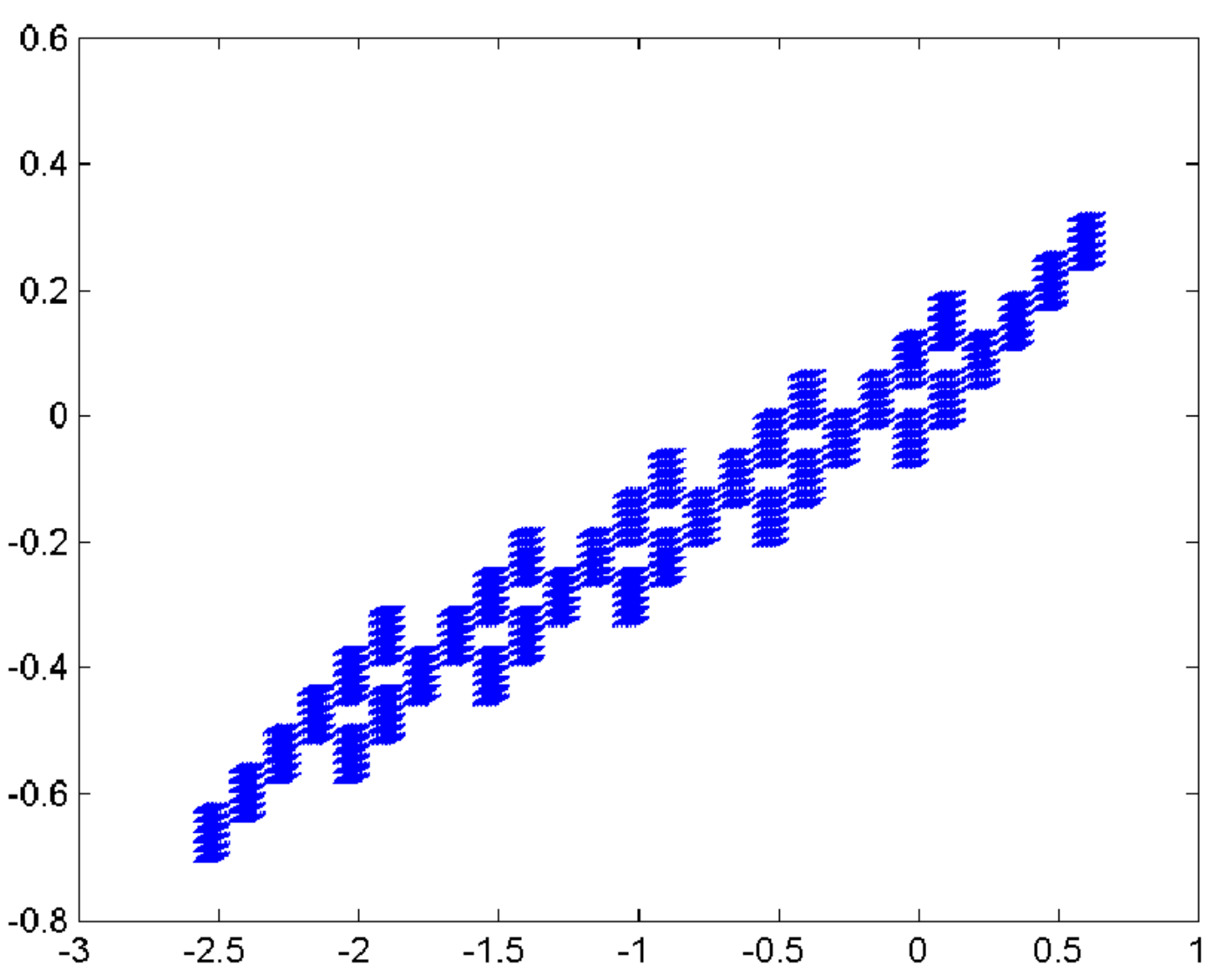}
}
\caption{(a) is disconnected and (b) is connected where  $A=[0,-8;1,-4]$, $v=(1,0)^t$.}\label{fig8}
\end{figure}

\begin{figure}[h]
 \centering
\subfigure[m=6]{
 \includegraphics[width=5cm]{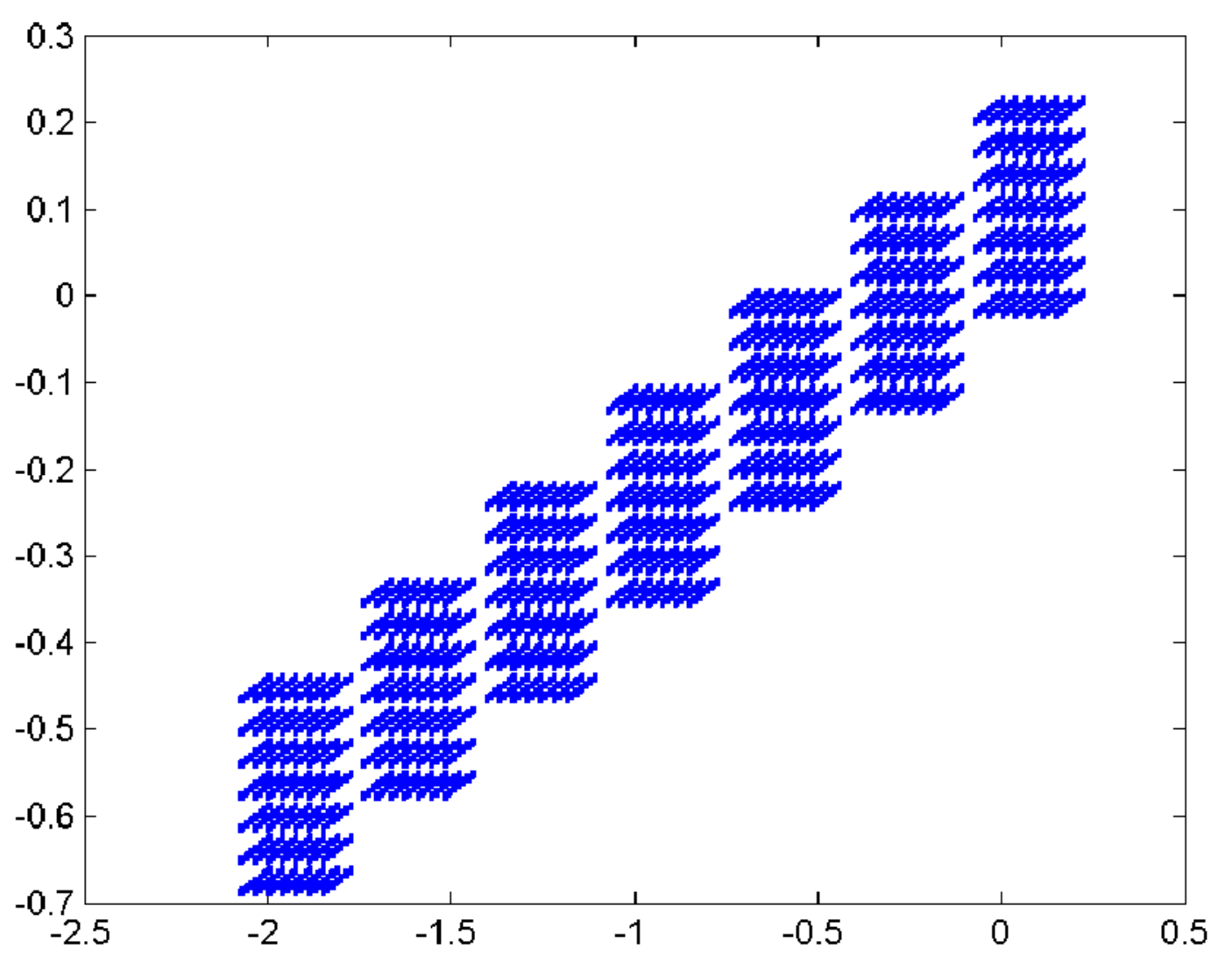}
 }\qquad
 \subfigure[m=7]{
 \includegraphics[width=5cm]{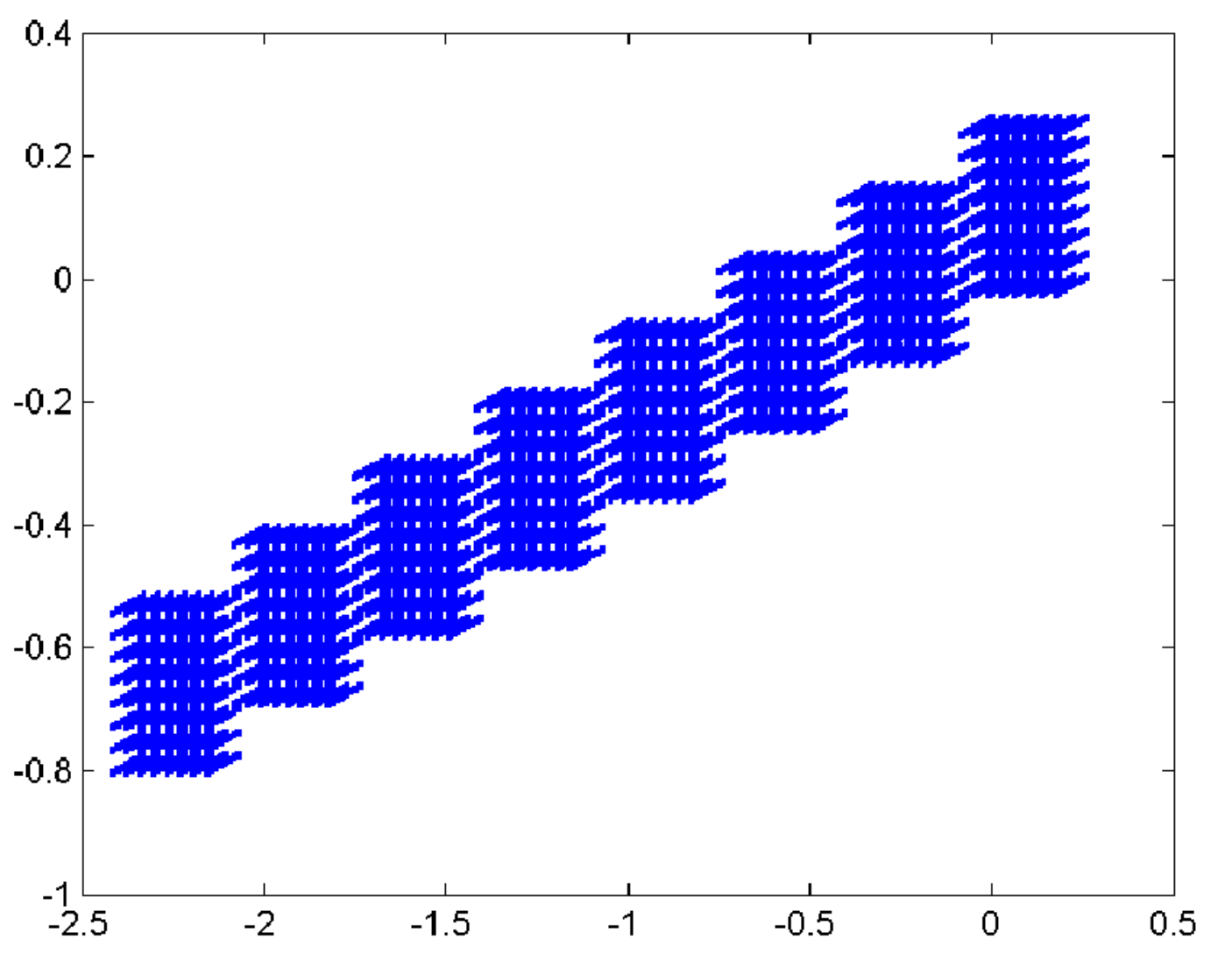}
}
\caption{(a) is disconnected and (b) is connected where $A=[0,-9;1,-3]$, $v=(1,0)^t$.}\label{fig9}
\end{figure}

\medskip

For other unsolved cases,  by observing computer graphs, we conclude with the following conjecture.

\begin{Coj}
Let the characteristic polynomial of $A$ be $f(x)=x^2+bx+c$ and a digit set $\mathcal{D}=\{0,1,\dots, m\}v$ where $m\geq 1$ and $v\in \mathbb{R}^2$ such that $\{v,  Av\}$ are linearly independent. Then

 \medskip
{\rm(i)}  if $c=|b|=4$, then $ T(A,\mathcal{D})$ is connected if and only if $m\geq 2$;
 \medskip

{\rm(ii)} otherwise, $ T(A,\mathcal{D})$ is connected if and only if
\begin{equation*}
m\geq \left\{
\begin{array}{ll}
\max\{c-|b|+1,|b|-1\} & \quad  c>0, |b|\geq 2 \\  \\
c-1 & \quad  c>0, |b|\leq 1 \\  \\
\max\{|c|-|b|-1,|b|+1\} & \quad  c<0.
\end{array}
\right.
\end{equation*}
\end{Coj}

\bigskip

\section{\bf Non-consecutive collinear digit set}

By previous Section 2, we know that if $A\in M_2({\mathbb Z})$ is an expanding matrix, then its characteristic polynomial has HRP. Let ${\mathcal{D}'}=\{0,1,\dots,(|\det(A)|-1)\}v$ be a consecutive collinear digit set with $\#{\mathcal D}'=|\det(A)|$. By Proposition \ref{thm2.2},
then the associated self-affine tile $T(A, {\mathcal D}')$ is always connected. However there  are few results on the non-consecutive collinear digit sets.
In \cite{LLu1}, Leung and one of the authors first study this case,  by checking $10$ eligible characteristic polynomials of the $A$ with $|\det A|=3$ case by case, they obtained a complete characterization for connectedness of $T(A,\mathcal{D})$ with $\mathcal{D}=\{0,1,m\}v$. In the section, we further study this kind of digit sets in more general situations. Suppose the characteristic polynomial of $A$ is of the form $f(x)=x^2-(p+q)x+pq$ where $|p|,|q|\geq 2$ are integers,  and $\mathcal{D}=\{0,1, \dots, |pq|-2, |pq|-1+s\}v$. By letting  $$f_1(x)=x^2\pm 4x+4 \quad\text{and}\quad f_2(x)=x^2\pm7x+12,$$ we have the following criterion for the connectedness.

\medskip

\begin{thm} \label{thm3.2}
Let the characteristic polynomial of $A$ be $f(x)=x^2-(p+q)x+pq$ and a digit set $\mathcal{D}=\{0,1, \dots, |pq|-2, |pq|-1+s\}v$ where $s\geq 0$, $|p|,|q|\geq 2$ are integers and $v\in \mathbb{R}^2$ such that $\{v,  Av\}$ are linearly independent. Then

 \medskip

{\rm(i)} if  $f\neq f_1, f_2$, then $T(A,\mathcal{D})$ is connected if and only if $s=0$;

\medskip
{\rm(ii)} if  $f= f_1$ or $f_2$ , then $T(A,\mathcal{D})$ is connected if and only if $s=0$ or $1$.
\end{thm}

\begin{proof}
If $s=0$, then  $ T(A, \mathcal{D})$ is always connected by Proposition \ref{thm2.2}.

\medskip

\noindent (i) Suppose $ T(A, \mathcal{D})$ is connected, then $(T+(|pq|-1+s)v)\cap (T+iv)\neq \emptyset$ for some $0\leq i\leq |pq|-2$. Let $r=|pq|-1+s-i$, then $rv\in T-T$, i.e., $rv=\sum_{i=1}^{\infty}b_iA^{-i}v$ where $b_i\in \Delta D$. By \eqref{(2.1)}, we obtain a new neighbor $T+l^*$ where $l^*=-(rpq+b_2)v +(r(p+q)-b_1)Av$.

 \medskip

If $pq>0$, then by (\ref{estimate}) and Lemma \ref{thm2.5}, we have
\begin{eqnarray}\label{3.8}
(1+s)|pq|-(|pq|-1+s)&\leq&  r|pq|-(|pq|-1+s)\leq  |rpq+b_2| \nonumber \\
&\leq &(|pq|-1+s )\frac{|p|+|q|-1}{(|p|-1)(|q|-1)};
\end{eqnarray}

\begin{align}\label{3.9}
(1+s)|p+q|-(|pq|-1+s)\leq |r(p+q)-b_1|\leq \frac{|pq|-1+s}{(|p|-1)(|q|-1)}.
\end{align}

It follows from (\ref{3.8}) that
\begin{align}\label{3.10}
s\leq \frac{|p|+|q|-2}{|pq|-|p|-|q|}.
\end{align}
 Let
\begin{align*}
t=\frac{|p|+|q|-1}{(|p|-1)(|q|-1)}=\frac{1}{|p|-1}+\frac{1}{|q|-1}+\frac{1}{(|p|-1)(|q|-1)}.
\end{align*}
It is easy to see that $t<1$ if $|p|,|q|\geq 4$ or  one of $|p|,|q|$  is equal to $3$ and the other one is larger than $5$. Therefore
$|p|+|q|-2 < |pq|-|p|-|q|$, and $s<1$, i.e., $s=0$. (see Figure \ref{fig4})

If  one of $|p|, |q|$ is equal to $2$ and the other one is larger than $3$, without loss of  generality, suppose $|p|=2$ and $|q|\geq 3$. From \eqref{3.9} we get
\begin{align}\label{3.11}
s\leq\frac{|q|^2-2|q|+2}{|q|^2-2}=1-\frac{2|q|-4}{|q|^2-2}<1.
\end{align} Hence $s=0$.

\medskip

If $pq<0$, analogous to the \eqref{3.8}, then
\begin{eqnarray}
(1+s)|pq|-(|pq|-1+s)\leq (|pq|-1+s)\frac{|p+q|+1}{|pq|-|p+q|-1}.
\end{eqnarray}
We have
\begin{align}
s\leq\frac{|p+q|}{|pq|-|p+q|-2}=1-\frac{|pq|-2|p+q|-2}{|pq|-|p+q|-2}.
\end{align}
Since $pq<0$,  without loss of generality, we let $|p|\geq|q|$, it follows that $|pq|-|p+q|-1=|pq|-(|p|-|q|)-1=(|p|+1)(|q|-1)\geq |p|+1>|p|-|q|+1=|p+q|+1$. Thus $|pq|-2|p+q|-2>0$, and $s=0$.

\medskip

\noindent (ii) If  $f= f_1$, then \eqref{3.11} implies that $s\leq 1$ (see Figure \ref{fig5}); if $f=f_2 $, then \eqref{3.10} implies that $s\leq 1$ (see Figure \ref{fig6}). Conversely, for $s=1$, let $\Delta D_1=\{0,\pm 1,\pm2, \pm3,\pm4\}$ and $\Delta D_2=\{0,\pm1,\dots, \pm 12\}$ and let $A_1$ and $A_2$ denote the matrices of $f_1$ and $f_2$  respectively. We only need to show  that $v, \ 2v\in T-T$ (see \cite{KL} or \cite{LLu1}). Let $\Delta D_1'=\{0,\pm1,\pm2\}$ and $\Delta D_2'=\{0,\pm1,\dots, \pm6\}$. By Theorem \ref{thm3.1}, there exist sequences $\{b_{1i}\}_{i=1}^\infty$ where $b_{1i}\in\Delta D_1'$ and $\{b_{2i}\}_{i=1}^\infty$ where $b_{2i}\in\Delta D_2'$ such that $v=\sum_{i=1}^{\infty}b_{1i}A_1^{-i}v\in T-T$ and $v=\sum_{i=1}^{\infty}b_{2i}A_2^{-i}v\in T-T$. Moreover, $2b_{1i}\in \Delta D_1$ and $2b_{2i}\in \Delta D_2$. Hence $2v\in T-T$ as well.
\end{proof}

\begin{figure}[h]
 \centering
\subfigure[s=0]{
 \includegraphics[width=5cm]{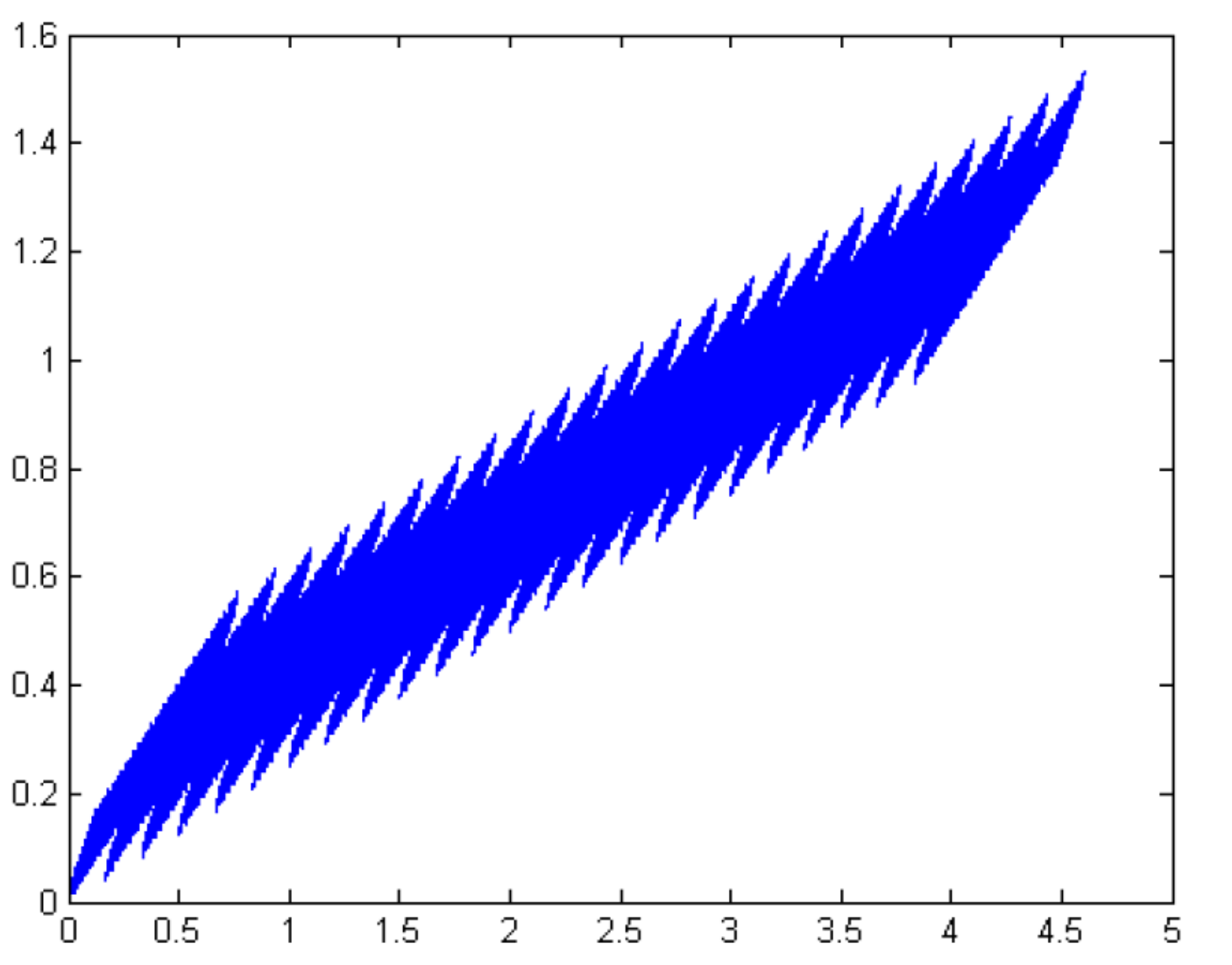}
 }\qquad
 \subfigure[s=1]{
 \includegraphics[width=5cm]{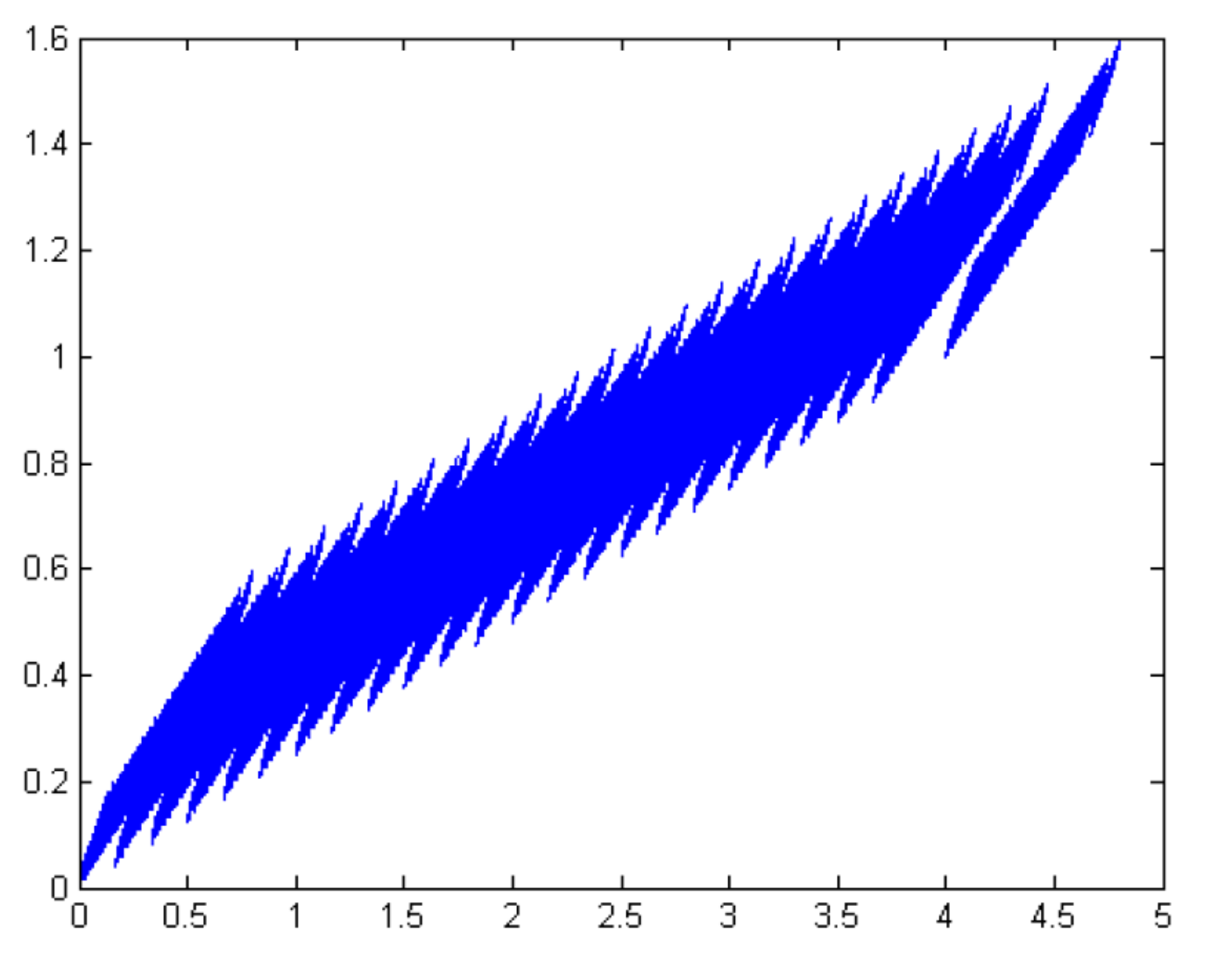}
}
\caption{(a) is connected and (b) is disconnected where $A=[6,0;-1,4], v=(1,0)^t$.}\label{fig4}
\end{figure}

\begin{figure}[h]
 \centering
\subfigure[s=1]{
 \includegraphics[width=5cm]{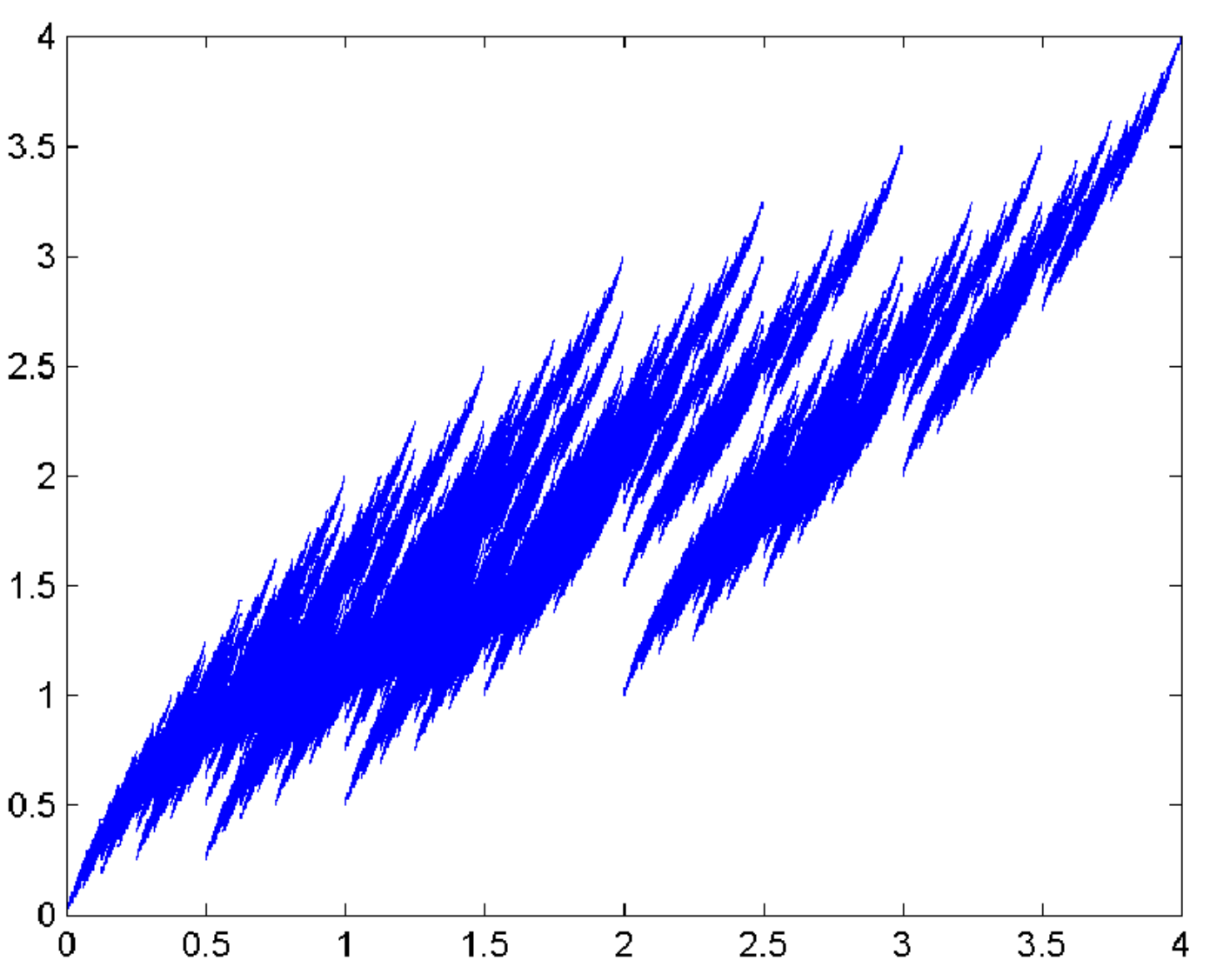}
 }\qquad
 \subfigure[s=2]{
 \includegraphics[width=5cm]{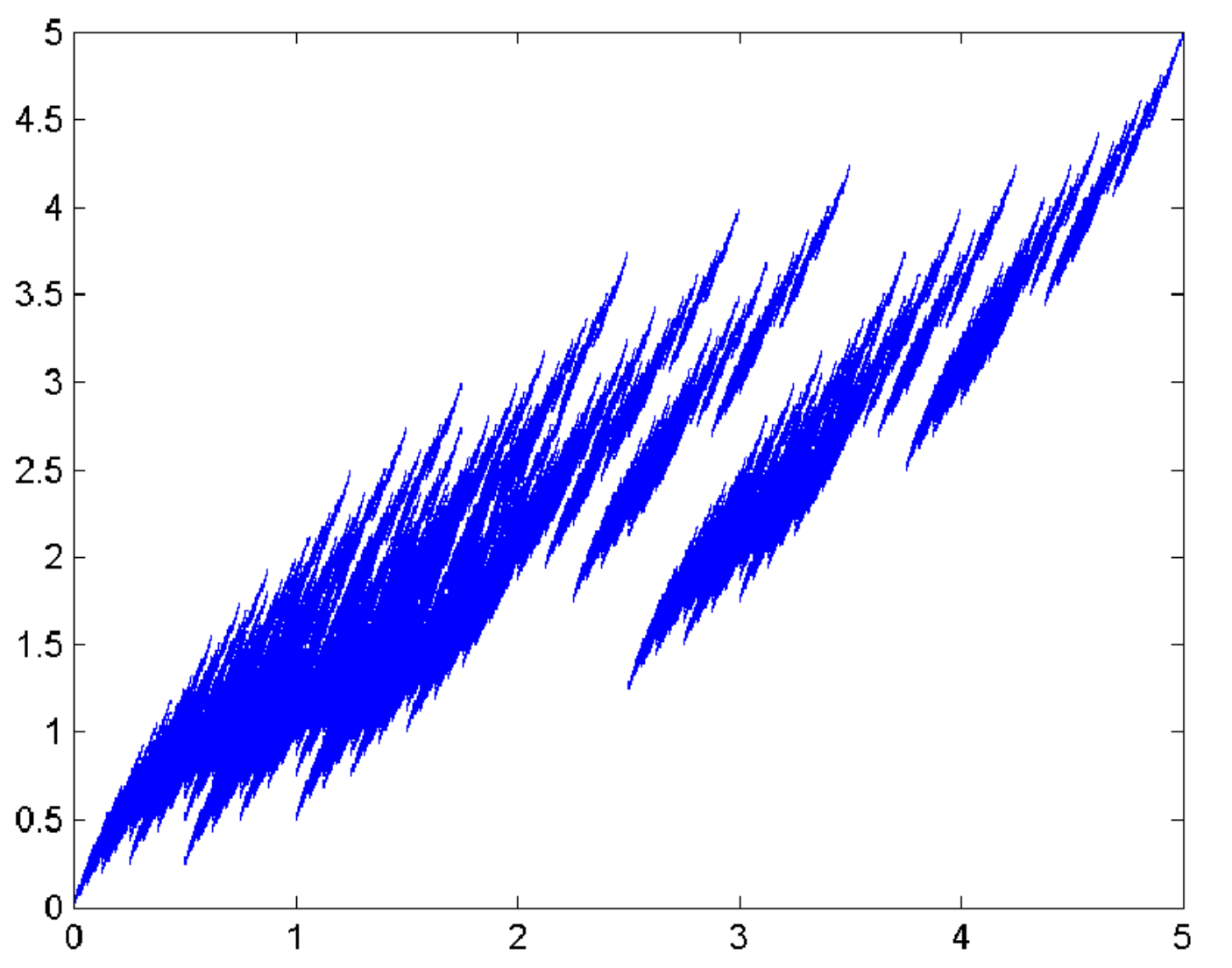}
}
\caption{(a) is connected and (b) is disconnected where $A=[2,0;-1,2], v=(1,0)^t$.}\label{fig5}
\end{figure}

\begin{figure}[h]
 \centering
\subfigure[s=1]{
 \includegraphics[width=5cm]{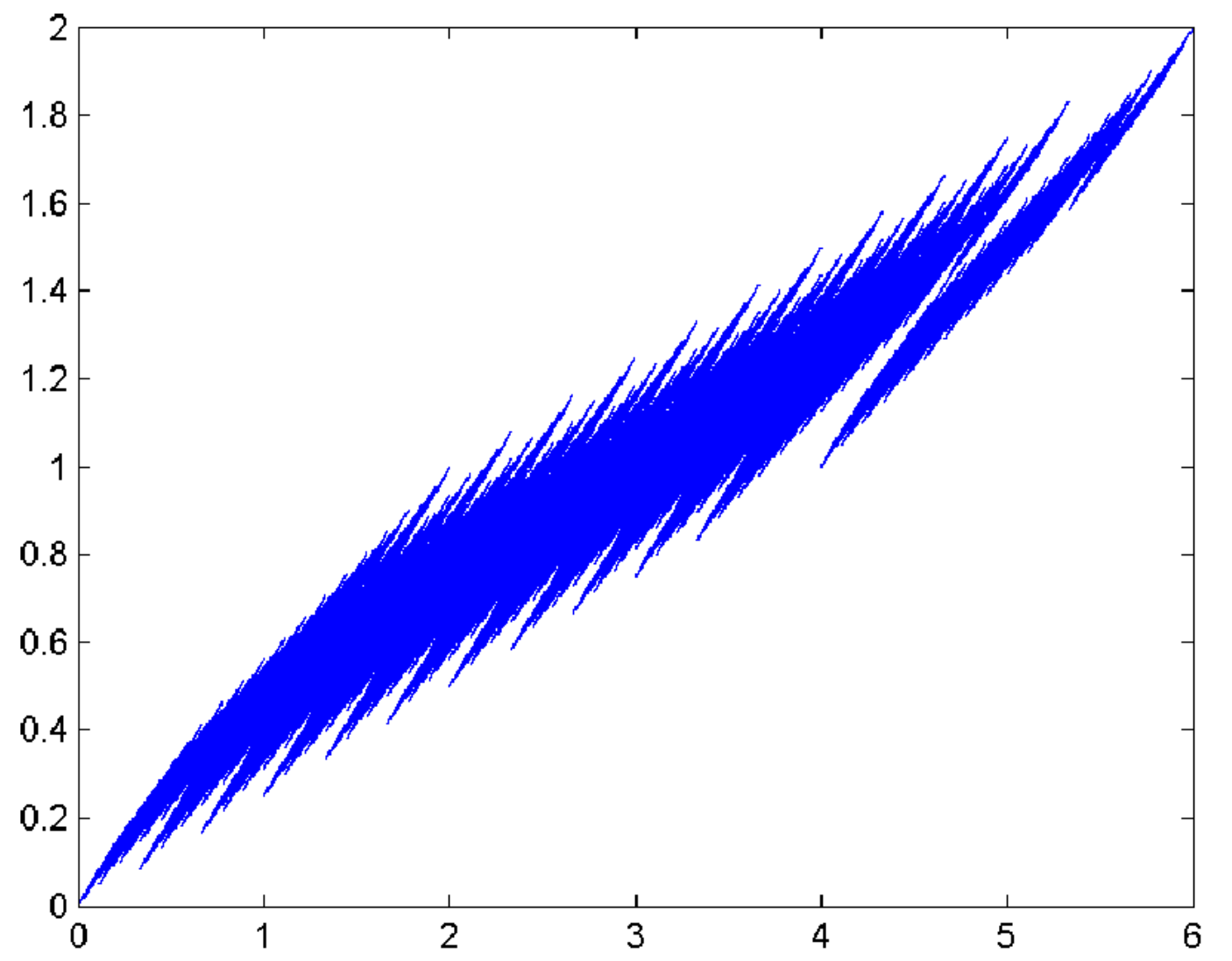}
}\qquad
 \subfigure[s=2]{
 \includegraphics[width=5cm]{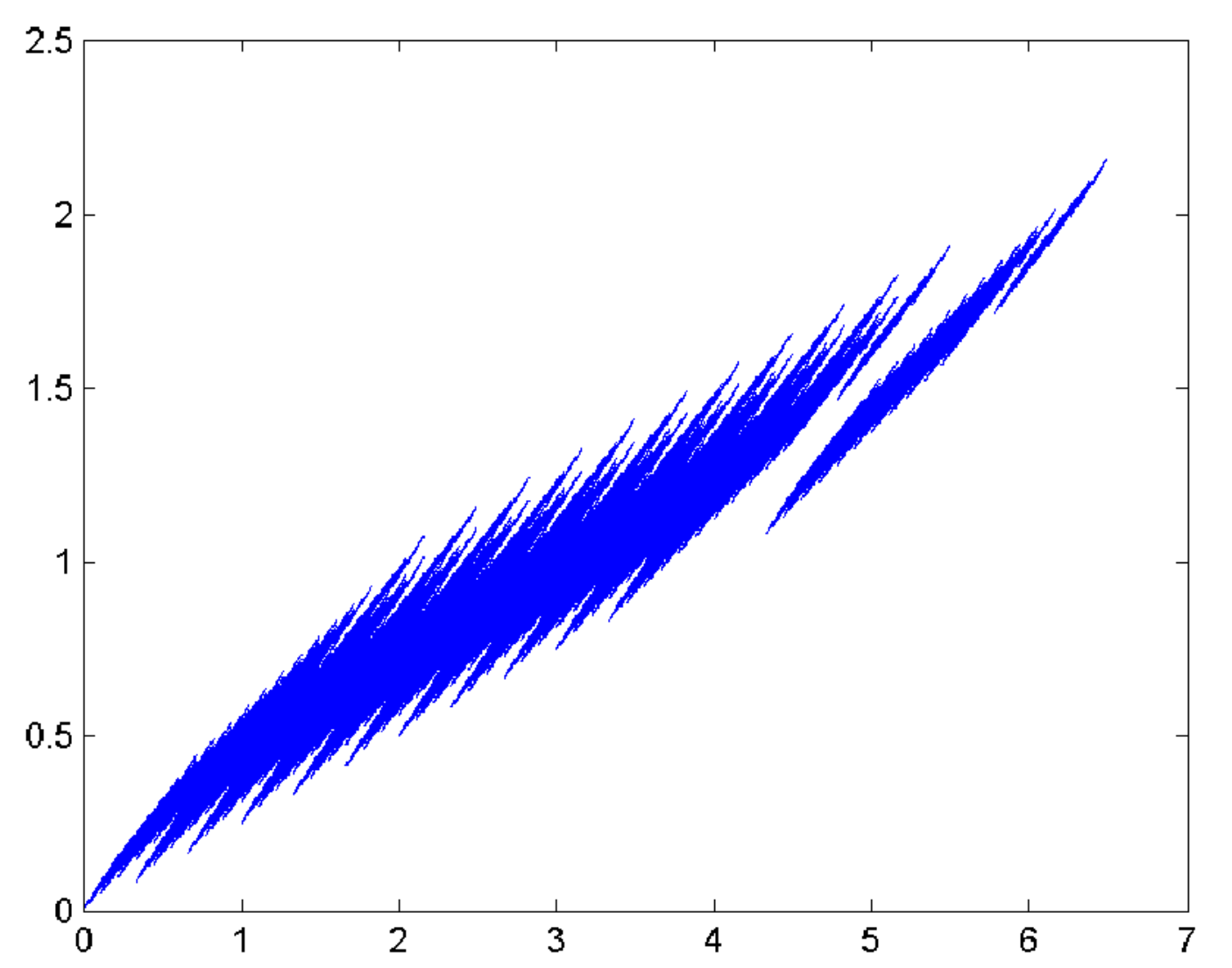}
}
\caption{(a) is connected and (b) is disconnected where $A=[3,0;-1,4], v=(1,0)^t$.}\label{fig6}
\end{figure}

\end{document}